\documentclass[english, 12pt]{article}

\usepackage{amsthm}
\newtheorem{theorem}{Theorem}
\newtheorem{lemma}{Lemma}
\newtheorem{prop}{Proposition}

\usepackage{amssymb,amsfonts,amsmath,mathrsfs}
\usepackage{epsfig}
\usepackage{xcolor}
\usepackage{graphicx}
\usepackage{float, wrapfig}
\usepackage{verbatimbox}
\usepackage{listings}

\usepackage[utf8]{inputenc}
\usepackage[T1]{fontenc}
\usepackage{babel}
\usepackage[margin=1in]{geometry}
\usepackage{hyperref}
\hypersetup{
    colorlinks=true,
    linkcolor=blue,
    filecolor=magenta,      
    urlcolor=cyan,
}

\usepackage{csquotes}

\allowdisplaybreaks

\begin{document}

\title{
On the Lorenz '96 Model and Some Generalizations
}
\author{John Kerin\footnote{email: \texttt{jkerin21@gmail.com}} , Hans Engler\footnote{Department of Mathematics and Statistics, Georgetown University, Washington, DC 20057, USA, email: \texttt{engler@georgetown.edu}} \footnote{Corresponding author}}

\maketitle

\begin{abstract}
In 1996,  Edward Lorenz introduced  a system of ordinary differential equations that describes a single scalar quantity as it evolves on a circular array of sites, undergoing forcing, dissipation, and rotation invariant advection. Lorenz constructed the system as a test problem for numerical weather prediction. Since then, the system has also found widespread use as a test case in data assimilation. Mathematically, it belongs to a class of dynamical systems with a single bifurcation parameter (rescaled forcing) that undergoes multiple bifurcations and  exhibits chaotic behavior for large forcing. In this paper, the main characteristics of the advection term in the model are identified and used to describe and classify a number of possible generalizations of the system. A graphical method to study the bifurcation behavior of constant solutions is introduced, and it is shown how to use the rotation invariance to compute normal forms of the system analytically. 
Problems with site-dependent forcing, dissipation, or advection  are considered and basic existence and stability results are proved for these extensions. 
We address some related topics in the appendices, wherein the Lorenz '96 system in Fourier space is considered, explicit solutions for some advection-only systems are found, and it is demonstrated how to use advection-only systems to assess numerical schemes.

\end{abstract}

\section{Introduction\label{s-intro}}
\label{sec-1}

The Lorenz '96 (L96) system was introduced by Edward Lorenz in 1996 \cite{Lorenz96} as a toy model for studying predictability, especially in weather and atmospheric systems. It is most often used as a test case for new data assimilation and ensemble forecasting techniques \cite{LorenzEmanuel98}, but has also been used to study some general aspects of chaos, turbulence, and linear response theory \cite{lucarini2011,gallavotti2014}. For these applications, the system has been considered mostly far into the chaotic regime in order to replicate properties of turbulent flow. However, there is a rich bifurcation structure as the system approaches chaos. Until recent years, there was little in the literature related directly to the mathematics of the model itself \cite{VanKekem2018wave}. As recognized by Lorenz and subsequent researchers, the L96 system lends itself to a number of modifications \cite{lorenz2005,wilks2005,gallavotti2014}. The main purpose of this paper is to introduce and study some new modifications motivated by fundamental properties of the system. This will lead to broad classes of L96-like systems for which we present some analytical and numerical tools.

\subsection{The L96 System}

The model is a system of coupled ordinary differential equations, which describes the transfer of some scalar atmospheric quantity. It includes three characteristic processes of atmospheric systems: advection, dissipation, and external forcing. The model is defined as follows. Consider a circular array of $N$ sites that are labeled $i = 0, 1, 2, \dots, N-1$. The labels $i \le -1$ and $i \ge N$ are extended periodically, identifying site $0$ with site $N$, site $-1$ with site $N-1$, and so on. Associated with site $i$ at time $T$ is a quantity $X_i(T)$. These quantities are governed by the system of differential equations
\begin{equation}
\frac{d}{dT}X_i(T) = \alpha X_{i-1}(X_{i+1} - X_{i-2}) - \beta X_i + \gamma
\label{e-lorenz-0}
\end{equation}
where $\alpha, \, \beta, \, \gamma$ are fixed positive constants. In the standard interpretation of the system, the sites are imagined to be equally spaced points circling the globe latitudinally. The quantity in question is added to each site at the rate $\gamma$ per time unit, via the \textit{forcing} term. During short time intervals of length $\Delta T$, an approximate fraction $\beta \Delta T$ of the quantity present at a given site is destroyed or dissipated, via the \textit{dissipation} term, and an amount that is proportional to $X_{i+1}(T)X_{i-1}(T)\Delta T$ enters site $i$ and an amount proportional to  $X_{i-2}(T)X_{i-1}(T) \Delta T$  leaves site $i$, with common proportionality factor $\alpha$, via the \textit{advection} term.

The effective number of parameters may be reduced by rescaling in the following way. Let
\begin{equation}
x_i(T) = \lambda X_i(T), \quad t = \sigma T \, .
\label{e-rescaling}
\end{equation}
Then by choosing $\sigma = \beta, \, \lambda = \frac{\alpha}{\beta}$ and setting $F = \frac{\alpha \gamma}{\beta^2}$ we arrive at the system of equations
\begin{equation}
\frac{d}{dt}x_i(t) = x_{i-1}(t)(x_{i+1}(t) - x_{i-2}(t)) - x_i(t) + F \, .
\label{e-lorenz-1}
\end{equation}
This is the original system considered in \cite{Lorenz96}. Writing $\mathbf{x} = (x_0, \, x_1, \dots, x_{N-1})$, considered as a column vector,  and $\mathbf{e} = (1, 1, \dots, 1)$, the system may be written as
\begin{equation}
\dot{\mathbf{x}} = G_L(\mathbf{x}) - \mathbf{x}  + F \mathbf{e}
\label{e-lorenz-1v}
\end{equation}
where the dot denotes differentiation with respect to $t$ and the mapping $G_L: \mathbb{R}^N \to \mathbb{R}^N$ is defined by
\begin{equation}
G_L(\mathbf{x})_i = x_{i-1}(x_{i+1} - x_{i-2}) \, .
\label{e-lorenz-advection}
\end{equation} 

It is easy to see that constant functions $\mathbf{x}(t) = F \mathbf{e}$ always solve Eq.~(\ref{e-lorenz-1v}), and it is known that these solutions are stable for $- \frac{1}{2} < F < \frac{8}{9}$. As $F$ increases, solutions with spatial and temporal periodicity appear. This is illustrated for $F=2$ in Fig.~\ref{fig-hov-f2}. Figure~\ref{fig-hov-f2}a shows the solution for $N = 36$ at $t = 500$, starting from random initial data and resulting in a periodic spatial pattern with spatial period 9. Fig.~\ref{fig-hov-f2}b of $t \mapsto x_0(t)$ shows that the solution is in fact periodic in time with period near 4. The right panel is a Hovmoeller plot of the solution for $500 \le t \le 510$, where sites are plotted horizontally, time increases in the vertical direction, and the solution value at each site and time is shown according to the color key on the far right. The solution is interpolated with a cubic spline. The plot shows that spatial and temporal periodicity come from regularly spaced waves that move to the left (``westward'') at a speed of about 1.2 sites per time unit. 

\begin{figure}[ht]
    \centering
        \includegraphics[width=0.49\textwidth]{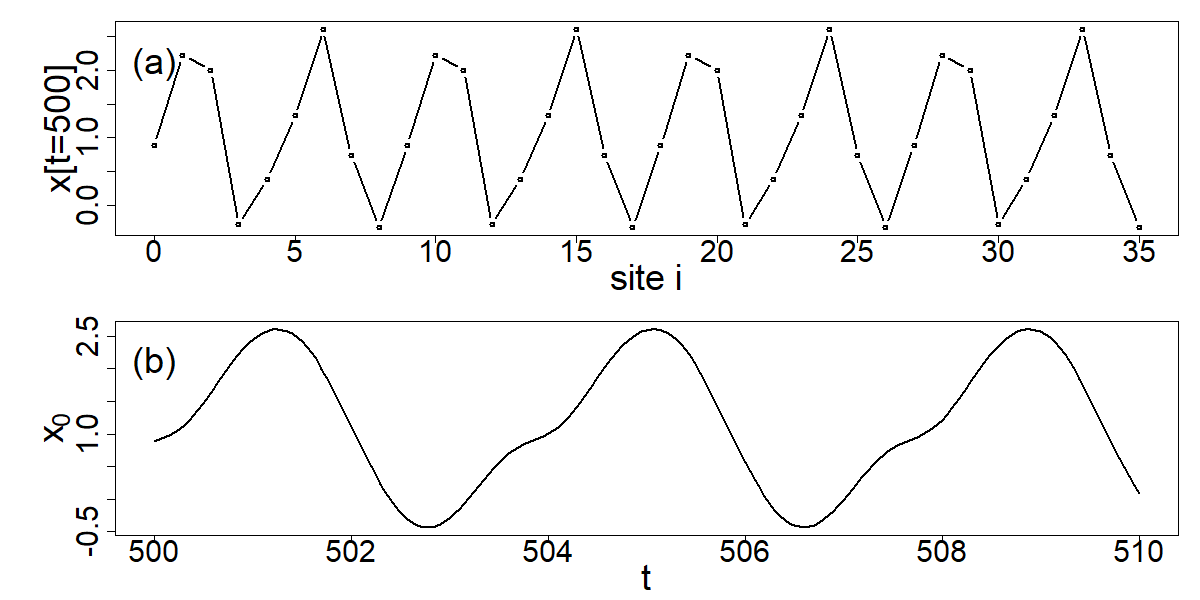}
        \includegraphics[width=0.49\textwidth]{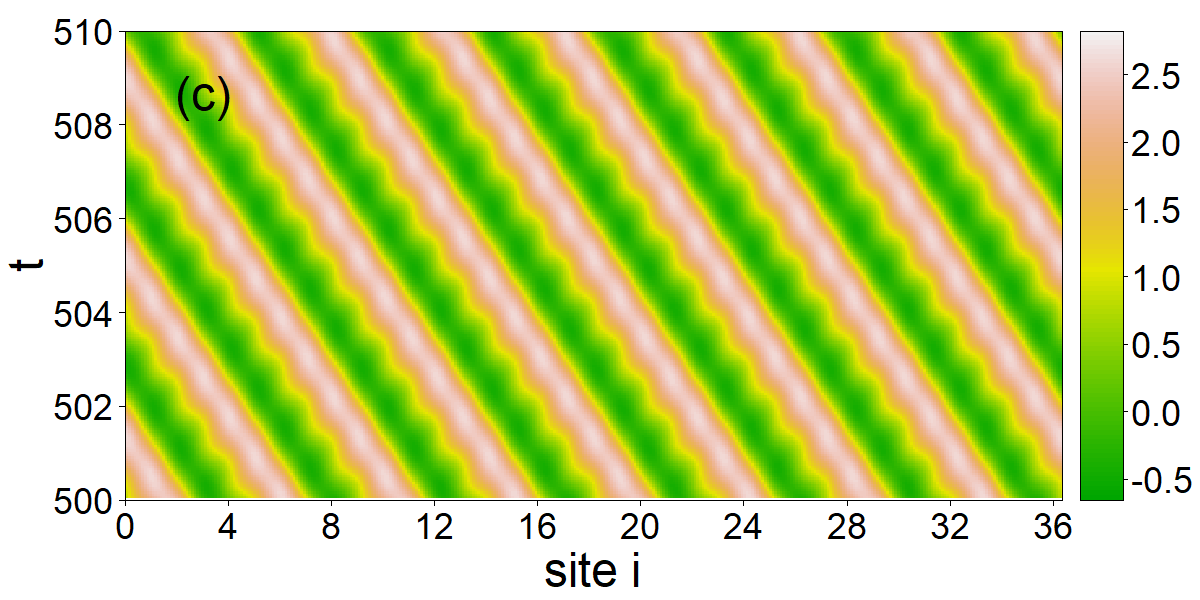}
        \caption{A solution of the L96 system with $F=2$ starting from random initial data. (a) All 36 sites at $t=500$. (b) First site for $500 \leq t \leq 510$. (c) Hovmoeller plot for $500 \leq t \leq 510$.}
        \label{fig-hov-f2}
\end{figure}

As $F$ increases further, the system becomes chaotic. The waves eventually break up and the behavior at any given site remains irregular even after a very long time. This is illustrated for $F=8$ in Fig.~\ref{fig-hov-f8}, well into the chaotic regime. Figure~\ref{fig-hov-f8}a again shows the solution for $N = 36$ at $t = 500$, starting from random initial data. Figure~\ref{fig-hov-f8}b shows the solution $x_0(t)$ for $500 \le t \le 510$. The Hovmoeller plot on the right for $500 \le t \le 510$ shows that there are still short lived spatially coherent wave patterns that move to the left at speeds between about 2 to 4 sites per time unit.

\begin{figure}[ht]
    \centering
        \includegraphics[width=0.49\textwidth]{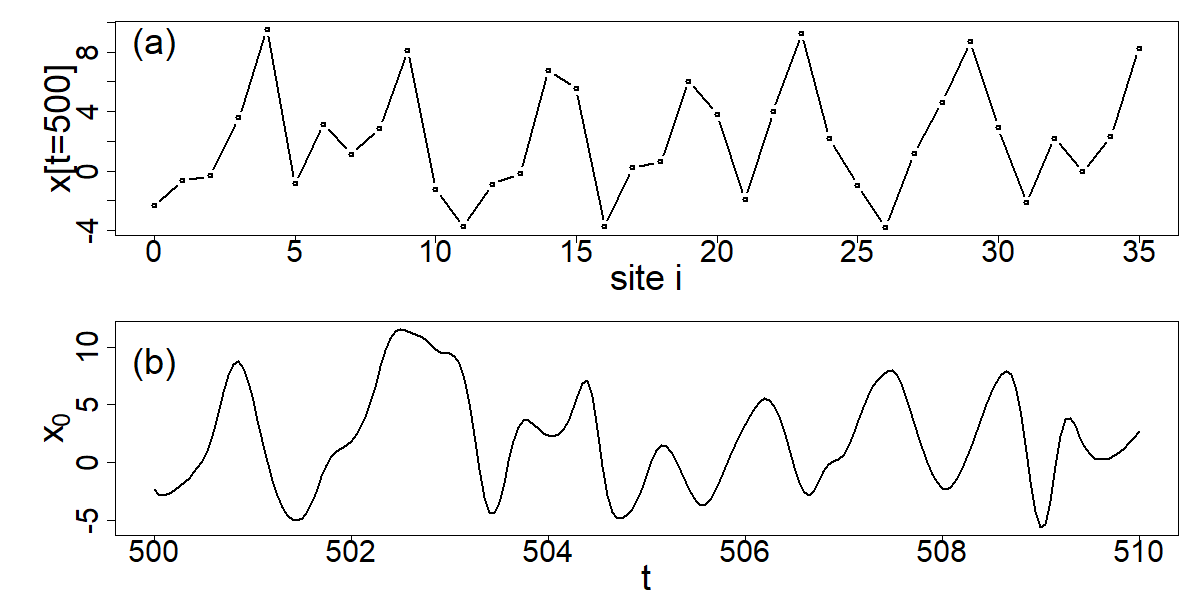}
        \includegraphics[width=0.49\textwidth]{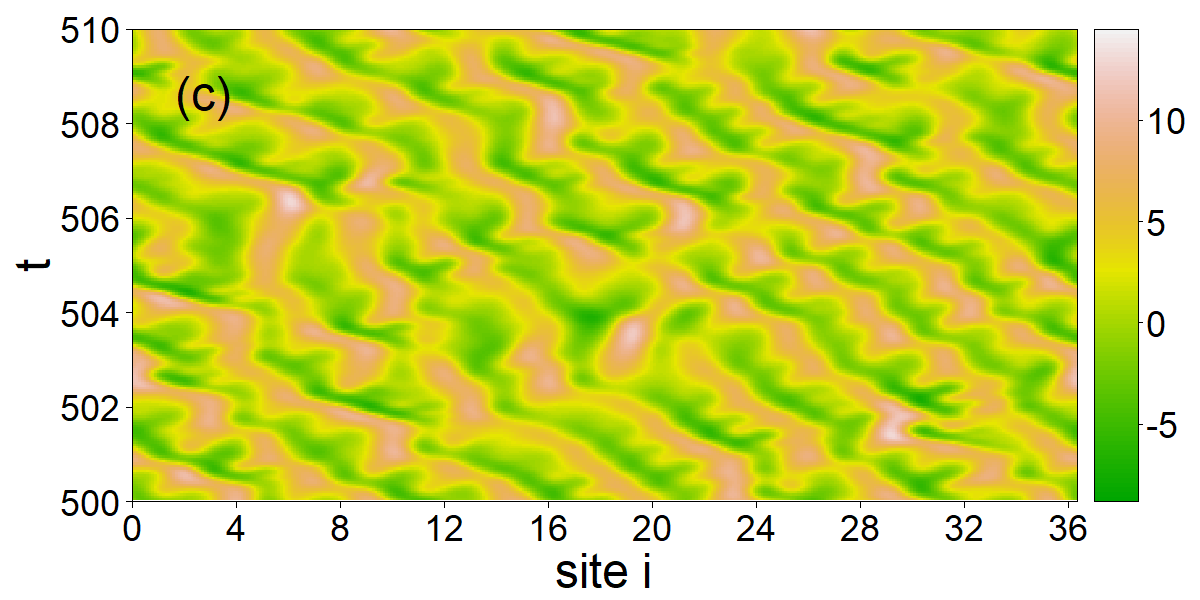}
        \caption{A solution of the L96 system with $F=8$ starting from random initial data. (a) All 36 sites at $t=500$. (b) First site for $500 \leq t \leq 510$. (c) Hovmoeller plot for $500 \leq t \leq 510$.}
        \label{fig-hov-f8}
\end{figure}

\subsection{Overview of the Paper}

In this paper, we generalize the L96 system in two ways, first by considering new advection terms, then by allowing for site-dependent parameters. We provide some broadly applicable tools to study the bifurcations and limit cycles of such systems. Some specific systems, including the original L96 system, are considered in more detail.

The paper is organized as follows. In Section~\ref{sec-2}, we discuss general classes of quadratic advection terms that share important properties with the term $G_L$ defined in Eq.~(\ref{e-lorenz-advection}). We characterize all such terms and compute their linearizations about constant solutions for later use. In Section~\ref{sec-3}, we show how to describe bifurcations from such constant solutions, and we prove that the first Hopf bifurcation for the L96 system is always supercritical, using a general approach for computing normal form coefficients. Section~\ref{sec-4} discusses multiple stable limit cycles of periodic solutions. Following the approach in \cite{VanKekem2018travelling}, we show how this can be explained by embedding the system in a suitable two parameter family of systems. As our main contribution, we  demonstrate how all normal form coefficients of the system for any number of sites may be computed exactly, which reveals the presence of a Neimark-Sacker (N-S) bifurcation. In Section~\ref{sec-5}, we turn to systems with site-dependent advection, dissipation, and forcing. We give general existence proofs for static and dynamic problems, show that for a large class of such problems the constant solution is always globally asymptotically stable for small forcing, and give some numerical results for the cases of site-dependent advection and dissipation. In Section~\ref{sec-6} we give conclusions and list some open questions.

Finally, we include a few appendices. Appendix A summarizes aspects of L96-like systems in Fourier space, which are useful for the normal form analysis. In Appendix B we compare the behavior of mid accuracy ODE solvers (e.g. Runge-Kutta) to that of high accuracy solvers with step size control by examining the computed energy of solutions of systems that have only advection. Appendix C is devoted to the special case of problems that have only advection, but no dissipation or forcing. Such systems are of interest because they have more conserved quantities than the full L96 system. For a system with symmetrized advection we find closed form solutions for $N = 4$ and $N = 6$.

\section{Advection Terms}
\label{sec-2}

In this section, $G:\mathbb{R}^N \to \mathbb{R}^N$ denotes an arbitrary continuously differentiable mapping, while $G_L:\mathbb{R}^N \to \mathbb{R}^N$ denotes the Lorenz advection term that is defined in Eq.~(\ref{e-lorenz-advection}). We identify and examine some properties of $G_L$ and identify other mappings $G$ that share these properties, described in \textbf{(i)}-\textbf{(iv)} in Section~\ref{sec-2-adv-properties} below. 

\subsection{Properties of Suitable Advection Terms}
\label{sec-2-adv-properties}

\noindent \textbf{(i)} The mapping $G_L$ is \textbf{quadratic}. A mapping $G$ has this property if it may be written as 
\begin{equation}
G(\mathbf{x}) = \tfrac12 \mathcal{B}(\mathbf{x}, \mathbf{x})
\label{e-quadratic}
\end{equation}
where $\mathcal{B}: \mathbb{R}^N \times \mathbb{R}^N \to \mathbb{R}^N$ is symmetric and bilinear. The bilinear map $\mathcal{B}$ may be recovered from $G$ by the formula
$$
\mathcal{B}(\mathbf{x}, \mathbf{y}) =  G(\mathbf{x} + \mathbf{y}) - G(\mathbf{x} - \mathbf{y}) \, .
$$
Specifically, in the L96 case the bilinear map is given by
\begin{equation}
\mathcal{B}_L (\mathbf{x}, \mathbf{y})_j = x_{j-1}(y_{j+1} - y_{j-2}) + y_{j-1}(x_{j+1} - x_{j-2}) \, .
\label{e-bilinear}
\end{equation}
Consequently, the Taylor expansion of a quadratic mapping $G$ about any $\mathbf{x}_0 \in \mathbb{R}^N$ has the form
\begin{equation}
G(\mathbf{x}_0 + \mathbf{y}) = G(\mathbf{x}_0)  +  A[\mathbf{x}_0]\mathbf{y} +  G(\mathbf{y}) 
\label{e-quadratic.1}
\end{equation}
where now $A$ is a linear map from $\mathbb{R}^N$ to the set of linear  self-maps of $\mathbb{R}^N$, defined by
\begin{equation}
A[\mathbf{x}] = \mathcal{B}(\mathbf{x}, \cdot) \, .
\label{e-linearization}
\end{equation}

\medskip
\noindent \textbf{(ii)} Further, the mapping $G_L$ is \textbf{energy-preserving}. A mapping $G$ has this property if for all $\mathbf{x}$
\begin{equation}
 \mathbf{x}^{\text{T}} G(\mathbf{x})  = 0  \, .
\label{e-energy.0}
\end{equation} 
This follows for $G_L$  by a direct calculation. We define the \textbf{energy} as $E = \mathbf{x}^{\text{T}} \mathbf{x}.$

\medskip
\noindent \textbf{(iii)} In addition, the mapping $G_L$ is \textbf{equivariant} with respect to the group  of coordinate rotations acting on $\mathbb{R}^N$. This group is generated by the rotation $\rho: \mathbb{R}^N \to \mathbb{R}^N$, i.e.  
\begin{equation}
    \rho(\mathbf{x}) = (x_1, \, x_2, \dots, \, x_{N-1}, \, x_0) \, .
    \label{e-rotation.0}
\end{equation}
A mapping $G$ from a suitable domain to itself is said to be equivariant with respect to a group $\mathcal{H}$ of transformations of its domain if for all group elements $h \in \mathcal{H}$  and all $\mathbf{x}$ in the domain
\begin{equation}
G(h (\mathbf{x})) = h \left(G(\mathbf{x})\right) \, .
\label{e-equivariant}
\end{equation}
We call the mapping $G_L$ $\langle \rho \rangle$\textbf{-equivariant}. Clearly, a quadratic mapping $G$ is $\langle \rho \rangle$-equivariant if and only if the associated bilinear map $\mathcal{B}$ from Eq.~(\ref{e-quadratic}) satisfies $\rho \mathcal{B}(\mathbf{x}, \mathbf{y}) = \mathcal{B}(\rho \mathbf{x}, \rho \mathbf{y})$ for all $\mathbf{x}, \,  \mathbf{y}$.
From now on, we use \textit{equivariant} to mean \textit{$\langle \rho 
\rangle$-equivariant}.

An equivariant mapping is defined by its behavior on a single element of each orbit of the group. Here, such a mapping may be described by its behavior at a single component. In particular, if $G$ is \textit{quadratic},
then we may write 
\begin{equation}
    G(\mathbf{x})_0 = \tfrac{1}{2} \mathcal{B}(\mathbf{x},\mathbf{x})_0 = \tfrac{1}{2} \sum\limits_{i,j} Q_{ij} x_i x_j = \tfrac{1}{2} \mathbf{x}^\text{T} Q \mathbf{x}, \quad \mathcal{B}(\mathbf{x}, \mathbf{y})_0 = \mathbf{x}^\text{T} Q \mathbf{y}
\label{e-define-Q}
\end{equation}
where $Q$ is a real symmetric matrix since $\mathcal{B}$ is symmetric. Identifying $\rho$ with its matrix representation, the equivariance of $G$ gives
\begin{equation}
    G(\mathbf{x})_m = \tfrac{1}{2} (\rho^m \mathbf{x})^\text{T} Q \rho^m \mathbf{x} = \tfrac{1}{2} \mathbf{x}^\text{T} (\rho^\text{T})^m Q \rho^m \mathbf{x} .
\end{equation}
When applied on the left of any matrix $M$, $\rho$, resp. $\rho^\text{T}$, shifts the rows of $M$ up, resp. down, by one row, circulantly. When applied on the right, it shifts the columns of $M$ right, resp. left, by one column, circulantly. Therefore $(\rho^\text{T})^m Q \rho^m$ is $Q$ shifted down and to the right, circulantly, by $m$. Since $Q$ is symmetric, $(\rho^\text{T})^m Q \rho^m$ is also symmetric.

If $G$ is \textit{equivariant} and \textit{energy-preserving}, then it must vanish on all constant input vectors $\lambda \mathbf{e}$, since $G(\lambda \mathbf{e})$ must be a multiple of $\mathbf{e}$ and $\mathbf{e}^{\text{T}} G(\lambda \mathbf{e}) = 0$. Additionally, its linearization at a constant input vector $\lambda \mathbf{e}$ is also equivariant, meaning it may be described by a \textit{circulant matrix}.

\medskip
\noindent \textbf{(iv)} Finally, each component of $G_L$ depends only on a few neighboring components  of $\mathbf{x}$. We say that the mapping $G$ is \textbf{$k$-localized} if for each $i$, $G(\mathbf{x})_i$ depends only on $x_{i-k}, \dots, x_{i+k}$, where indices are taken modulo $N$. For example, $G_L$ is 2-localized. An \textit{equivariant, quadratic} mapping $G$ may be described by its behavior at the 0-th component via Eq.~(\ref{e-define-Q}). If the mapping is \textit{$k$-localized}, then shifting $Q$ by $k$ columns to the right and $k$ rows down, circulantly, results in a matrix that is zero outside its top left $2k+1 \times 2k+1$ block.

For later use, we define the coordinate reflection map $\tau$ defined by
\begin{equation}
\tau (\mathbf{x}) = ( x_{N-1},x_{N-2}, \dots, x_1, x_0) .
\label{e-reflection.0}
\end{equation}
Then if a mapping $G$ is quadratic, energy-preserving, equivariant, or $k$-localized, then so is the mapping $\tilde G$, defined by $\tilde G(\mathbf{x}) = \tau \circ G(\tau (\mathbf{x}))$. When $G$ is replaced with $\tilde G$,  the direction of the advection term is simply reversed everywhere.

\subsection{Classification of Maps}
\label{sec-2-g-maps}

We now give a complete classification of all \textit{quadratic, energy-preserving, equivariant, 3-localized} mappings $G$.

Consider first the set of all quadratic equivariant maps $G$, which forms a vector space of dimension $\frac{N(N+1)}{2}$, as Eq.~(\ref{e-define-Q}) shows. If $G$ is also energy-preserving, then 
\begin{align*}
    0&=\sum\limits_k x_k \mathbf{x}^\text{T} (\rho^\text{T})^k Q \rho^k \mathbf{x}\\
    &=\sum\limits_{i,j,k} x_k x_i \left[ (\rho^\text{T})^k Q \rho^k \right]_{ij} x_j\\
    &=\sum\limits_{i,j,k} Q_{(i-k)(j-k)} x_i x_j x_k
\end{align*}
where $i,j,k$ run from 0 to $N-1$ and indices are taken modulo $N$. Since $Q$ is symmetric, this condition is equivalent to
\begin{equation}
    Q_{(i-k)(j-k)} + Q_{(i-j)(k-j)} + Q_{(j-i)(k-i)} = 0 \quad \forall \, i,j,k \in \lbrace 0,\hdots,N-1 \rbrace.
    \label{e-Q-sum-zero}
\end{equation}
Taking $k=0$ corresponds to $G(\mathbf{x})_0$; that is, $Q_{ij}$ is the coefficient of $x_i x_j$, giving
\begin{equation}
    Q_{ij} + Q_{(i-j)(-j)} + Q_{(j-i)(-i)} = 0.
\label{e-Q-energy}
\end{equation}

Next consider quadratic, equivariant maps $G$ that are also $k$-localized. Thus each component of $G(\mathbf{x})$ depends on up to $2k+1$ sites. These sites should all be different, therefore we require $N \ge 2k+2$.  For $k = 1, \, N \ge 4$, this is a 6-dimensional space, for $k=2, \, N \ge 6$ it is a 15-dimensional space, and for $k=3, \, N \ge 8$ it is a 28-dimensional space. If now $G$ is also energy-preserving, then Eq.~(\ref{e-Q-energy}) restricts the set of possible mappings further. The following result allows one to list all such maps for $k \le 3$. Clearly, any equivariant, quadratic, energy-preserving map must have at least two quadratic terms in each component. The mappings that are listed below are among the simplest possible in that they each use exactly two terms.  

\begin{theorem} 
\label{thm-all.quadratic}
Consider the set of all quadratic, energy-preserving, equivariant, $k$-localized maps $G: \mathbb{R}^N \to \mathbb{R}^N$. \\
a) If $k = 1$ and $N \ge 4$, this is a two-dimensional space, $\mathscr{G}_1$. A basis is given by $G_1, \tilde G_1$, defined by
\begin{equation}
G_1(\mathbf{x})_0 = x_1^2 - x_0x_{-1}, \quad \tilde G_1 = \tau \circ G_1 \circ \tau, \quad \tilde G_1(\mathbf{x})_0 = x_{-1}^2 - x_0x_1 .
\label{e-form-k1}
\end{equation}
b) If $k = 2$ and $N \ge 6$, this is a six-dimensional space, $\mathscr{G}_2$. A basis is given by $G_1, \, \tilde G_1$, and $G_2, \, \tilde G_2, \, G_3, \, \tilde G_3$, defined by
\begin{eqnarray}
\begin{aligned}
G_2(\mathbf{x})_0 &= x_2^2 - x_0x_{-2}, \quad \tilde G_2 =  \tau \circ G_2 \circ \tau \\
G_3(\mathbf{x})_0 &= x_{-1}x_1 - x_{-2}x_{-1}, \quad \tilde G_3 = \tau \circ  G_3 \circ \tau .
\end{aligned}
\label{e-form-k2}
\end{eqnarray}
c) If $k = 3$ and $N \ge 8$, this is a 12-dimensional space, $\mathscr{G}_3$. A basis is given by $G_1, \, \tilde G_1, \dots \, \tilde G_3$, and $G_4, \, \tilde G_4, \dots, \, \tilde G_6$, defined by
\begin{eqnarray}
\begin{aligned}
G_4(\mathbf{x})_0 &= x_3^2 - x_0x_{-3}, \quad \tilde G_4 = \tau \circ G_4 \circ \tau  \\
G_5(\mathbf{x})_0 &= x_2x_3 - x_{-2}x_{1}, \quad \tilde G_5 = \tau \circ G_5 \circ \tau \\
G_6(\mathbf{x})_0 &= x_1x_3 - x_{-1}x_{2}, \quad \tilde G_6 = \tau \circ G_6 \circ \tau .
\end{aligned}
\label{e-form-k3}
\end{eqnarray}

\end{theorem}

\begin{proof}
Start by noting that Eq.~(\ref{e-Q-energy}) implies $Q_{00} = 0$ and $Q_{rs} = 0$ if $|s - r| > k$ and $N \ge 2k+2$. \\
For $k=1$ this leaves the four possible nonzero coefficients $Q_{-1,-1}, Q_{-1,0} = Q_{0,-1}, Q_{0,1} = Q_{1,0}, \, Q_{1,1}$ which must satisfy the equations $Q_{1,1} + 2Q_{-1,0} = 0$ and $Q_{-1,-1} + 2Q_{0,1} = 0$. The maps $G_1$ and $\tilde G_1$ correspond to two independent solutions of this homogeneous system.\\
For $k = 2$ the same reasoning results in 11 nonzero coefficients and three homogeneous equations in addition to the two equations for $k=1$. The maps $G_1, \, \dots, \, \tilde G_3$ correspond to six independent solutions of this homogeneous system.\\
The same technique may be applied to the case $k=3$.
\end{proof}

Henceforth, we call any \textit{quadratic, energy-preserving, equivariant} map from $\mathbb{R}^N$ to itself a \textbf{$\mathscr{G}$-map}. A map in $\mathscr{G}_k$ as described in Theorem~\ref{thm-all.quadratic} is a \textbf{$k$-localized $\mathscr{G}$-map}. In the following sections, when it is clear that we are talking about a $\mathscr{G}$-map we sometimes use ``advection term'' interchangeably.

The map $G_3$ is the L96 term $G_L$. Also, $G_2$ is essentially the same as $G_1$ with interacting sites always at a distance of $k=2$ instead of $k=1$, and $G_4$ is essentially the same as $G_1$ with interaction at a distance of $k=3$. Some additional $\mathscr{G}$-maps that have only two quadratic terms may be generated from the list above. The main examples are 
\begin{eqnarray}
G_7 = G_3 - \tilde G_3, &\quad& G_7(\mathbf{x})_0 = x_1 x_2 - x_{-1}x_{-2}\\
G_8 = G_5 - \tilde G_6, &\quad& G_8(\mathbf{x})_0 = x_2 x_3 - x_{-1}x_{-3}, \quad 
\tilde G_8 = \tau \circ G_8 \circ \tau .
\label{e-more-k3}
\end{eqnarray}
Note that $\tau \circ G_7 \circ \tau = - G_7$ and therefore there is no reason to consider $\tilde G_7$ separately. 

There is no 1-localized $\mathcal{G}$-map with the Liouville property $\nabla_{\mathbf{x}} \cdot G(\mathbf{x}) = 0$ as the theorem shows. The 2-localized $\mathcal{G}$-maps with this property are precisely the linear combinations of $G_3$ and $\tilde G_3$. Equations with nonlinearity from this set, called the \textbf{Orszag-McLaughlin} family after \cite{orszag1980} where a special case was introduced, were studied in \cite{abramov2008}.

\section{Linearization and Eigenvalue Curves}
\label{sec-3}

Recall the definition of the left rotation $\rho$ in Eq.~(\ref{e-rotation.0}), identified with its matrix representation. This is a unitary circulant matrix, i.e. $\rho^{-1} = \rho^{\text{T}}$. 
\begin{equation}
    \rho \equiv
    \begin{pmatrix}
    0 & 1 & 0 & \hdots & 0\\
    0 & 0 & 1 & \hdots & 0\\
    \vdots & \vdots & \vdots & \ddots & \vdots\\
    0 & 0 & 0 & \hdots & 1\\
    1 & 0 & 0 & \hdots & 0
    \end{pmatrix}
\end{equation}
Matrix rows and columns are labeled from 0 to $N-1$. The eigenvalues of $\rho$ are $\omega_N^j,  \, j = 0, \dots, N-1$, where $\omega_N = e^{2 \pi i/N}$ is a primitive $N$-th unit root.

Throughout the section, advection terms $G$ will be assumed to be $\mathscr{G}$-maps. Let $F \in \mathbb{R}$. We are interested in the general system
\begin{equation}
\dot{\mathbf{x}} = G(\mathbf{x}) - \mathbf{x}  + F \mathbf{e}.
\label{e-lorenz-general}
\end{equation}
Applying Eq.~(\ref{e-quadratic.1}), it follows that the linearization of this system~(\ref{e-lorenz-general}) about the constant solution $\mathbf{x}_F = F \mathbf{e}$ is the system
\begin{equation}
\dot{\mathbf{y}} = (FA - I)  \mathbf{y}
\label{e-lorenz-2lin}    
\end{equation}
where $A$ is the linearization of $G$ about the constant vector $\mathbf{e}$. The elements of $A$ are $A_{ij} = 2 \sum_{k=0}^{N-1} Q_{(j-i)(k-i)}$, with $Q$ defined by Eq.~(\ref{e-define-Q}).

Since $G$ is equivariant, $A$ is circulant. Its eigenvalues are therefore of the form $\lambda_j = p_A(\omega_N^j), \, j = 0, \dots, N-1$, where $p_A(z) = \sum_{j=0}^{N-1} A_{0j}z^j = \sum_{j,k=0}^{N-1} Q_{jk} z^j$ is a polynomial with coefficients given in the first row of $A$. Since all arguments of $p_A$ are unit roots, we may replace positive powers $z^j$ with negative powers $z^{j-N}$ if $j \ge \frac{N}{2}$, resulting in a \textbf{Laurent polynomial}. All eigenvalues of $A$ lie on the image of the complex unit circle $\mathbb{S}^1$ under the map $z \mapsto p_A(z)$. We therefore also call these \textbf{eigenvalue curves}.

\subsection{Characterizing Eigenvalue Curves}
\label{sec-3-eig-curves}

We now give a table of these polynomials $p_A$ for the maps $G_1, \dots, G_8$ identified in Theorem~\ref{thm-all.quadratic} and also describe the shape of the image of the unit circle under $p_A$. These images are plotted in Fig.~\ref{f-curves-0} for some of the Laurent polynomials. We also list the type of bifurcations corresponding to each eigenvalue curve.

\begin{table}[ht]
\begin{center}
\addvbuffer[12pt]{\begin{tabular}{l|l|l|l|l}
    $\mathscr{G}$-map & Laurent polynomial $p_A(z)$ & Shape of $p_A(\mathbb{S}^1)$ & $F>0$ & $F<0$ \\\hline
    $G_1$ & $ - z^{-1} - 1 + 2z $ & ellipse & none & pitchfork\\
    $G_2$ & $ -z^{-2} - 1 + 2z^2 $ & ellipse & none & pitchfork\\
    $G_3$ & $ - z^{-2} + z$ & trefoil & Hopf & pitchfork\\
    $G_4$ & $ - z^{-3} -1 + 2 z^3$ & ellipse & none & pitchfork\\
    $G_5$ & $-z^{-2} - z + z^2 + z^3$ & butterfly & Hopf & Hopf\\
    $ G_6$ & $-z^{-1} + z - z^2 + z^3$ &  kidney & Hopf & pitchfork\\
    $G_7$ & $-z^{-2} - z^{-1} + z + z^2$ &  vertical line & none & none\\
    $G_8$ & $-z^{-3} - z^{-1} + z^2 + z^3$ & bee & pitchfork & pitchfork
\end{tabular}}
\label{table:laurent-0}
\caption{Description of the eigenvalue curves of the eight simplest 3-localized $\mathscr{G}$-maps identified in Section~\ref{sec-2-g-maps}. The two rightmost columns give the types of the first expected bifurcation for $F>0$ and $F<0$ as the magnitude of $F$ increases.}
\end{center}
\end{table}

Note that all these Laurent polynomials satisfy $p_A(1) = 0$, i.e. the sum of coefficients vanishes. This is a consequence of linearizing energy-preserving quadratic maps about the vector $\mathbf{e}$, which follows from Eq.~(\ref{e-Q-sum-zero}).

\begin{figure}[ht]
 \centering
  \includegraphics[width=.95\textwidth]{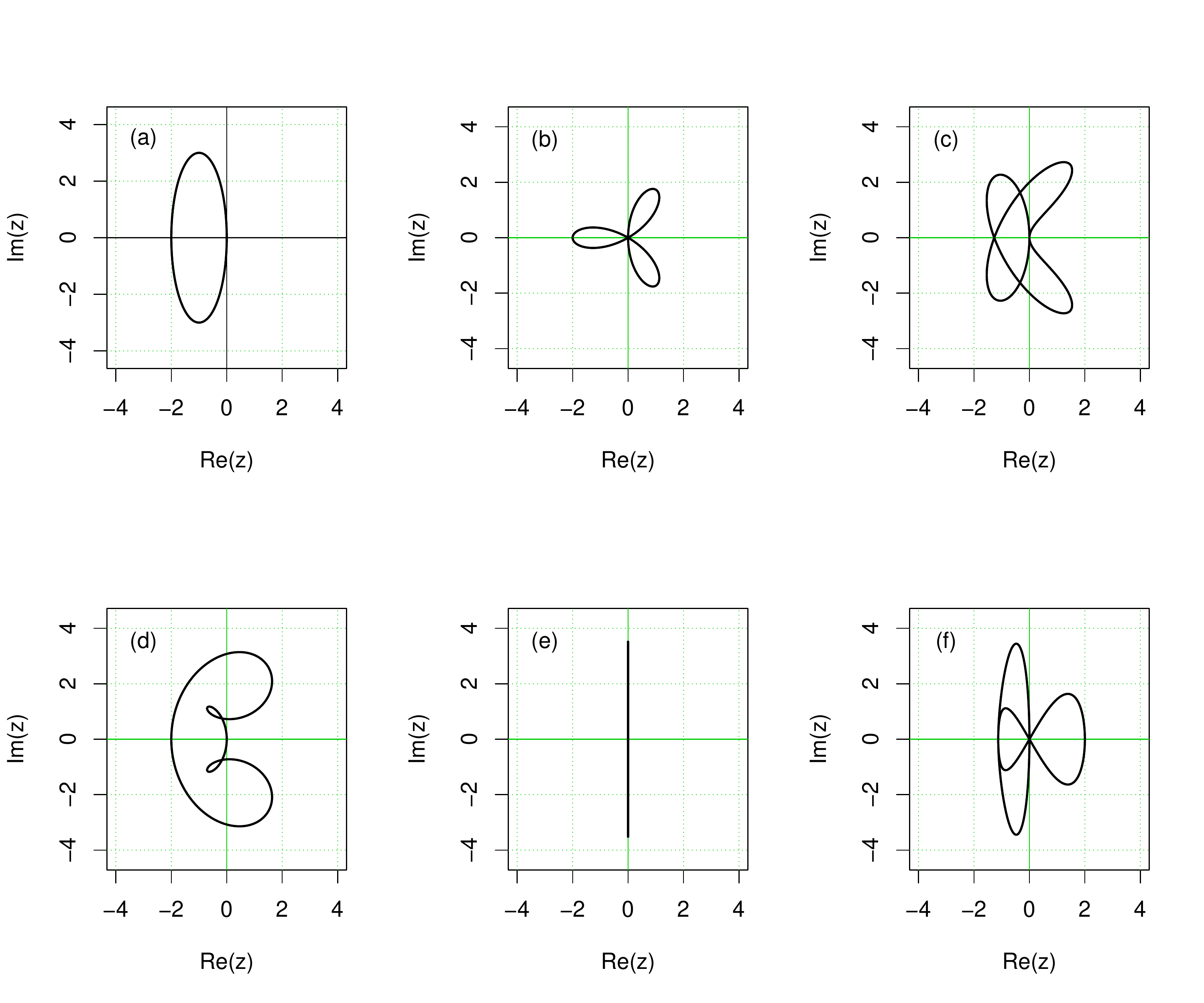}
  \caption{Images of the complex unit circle under the Laurent polynomials given in Table~1. (a): Ellipse ($G_1, \, G_2, \, G_4$). (b): Trefoil ($G_3$, L96 system). (c): Butterfly ($G_5$). (d): Kidney ($G_6$). (e): Vertical line ($G_7$). (f): Bee ($G_8$).}
  \label{f-curves-0}
\end{figure}

If a map $G_j$ is replaced by $\tilde G_j$, then $p_A(z)$ is replaced by $p_A(z^{-1})$. The image of the unit circle does  not change, but the curve is traversed in the opposite direction. If $G$ is replaced with $-G$, the image of the curve is reflected about the imaginary axis. The curves for $G_2$ and $G_4$ are the same as for $G_1$ but are traversed twice or three times as $z$ traverses $\mathbb{S}^1$.

For a concrete system, the eigenvalues of $A$ are a discrete set of points. For the case of the L96 system with $G = G_L = G_3$ and $N = 36$, this is illustrated in the left panel of 
Fig.~\ref{f-curves-1}, which shows the eigenvalues of the matrix $A$. The eigenvalues are the points $\lambda_j = \omega_{36}^j - \omega_{36}^{-2j}, \, \omega_{36} = e^{\pi i/18}$. We note that in this case $\lambda_0 = \lambda_{12} = \lambda_{24} = 0$ and that the single eigenvalue with smallest real part is $\lambda_{18} = -2$. The pair of eigenvalues with largest real part is $(\lambda_8, \, \lambda_{28})$ with $\Re \lambda_8 = \Re \lambda_{28} = \cos \frac{\pi}{9} + \sin \frac{\pi}{18} \approx 1.11334$. The pair
$(\lambda_7, \, \lambda_{29})$ with $\Re \lambda_7 = \Re \lambda_{29} = \cos \frac{2\pi}{9} + \sin \frac{\pi}{9} \approx 1.10806$ is close behind. There are four eigenvalues with real part 1, namely $\lambda_6 = 1 + \sqrt{3} \cdot i, \, \lambda_9 = 1 + i$, and their conjugates $\lambda_{27}$ and $\lambda_{30}$.

\begin{figure}[ht]
 \centering
  \includegraphics[width=.95\textwidth]{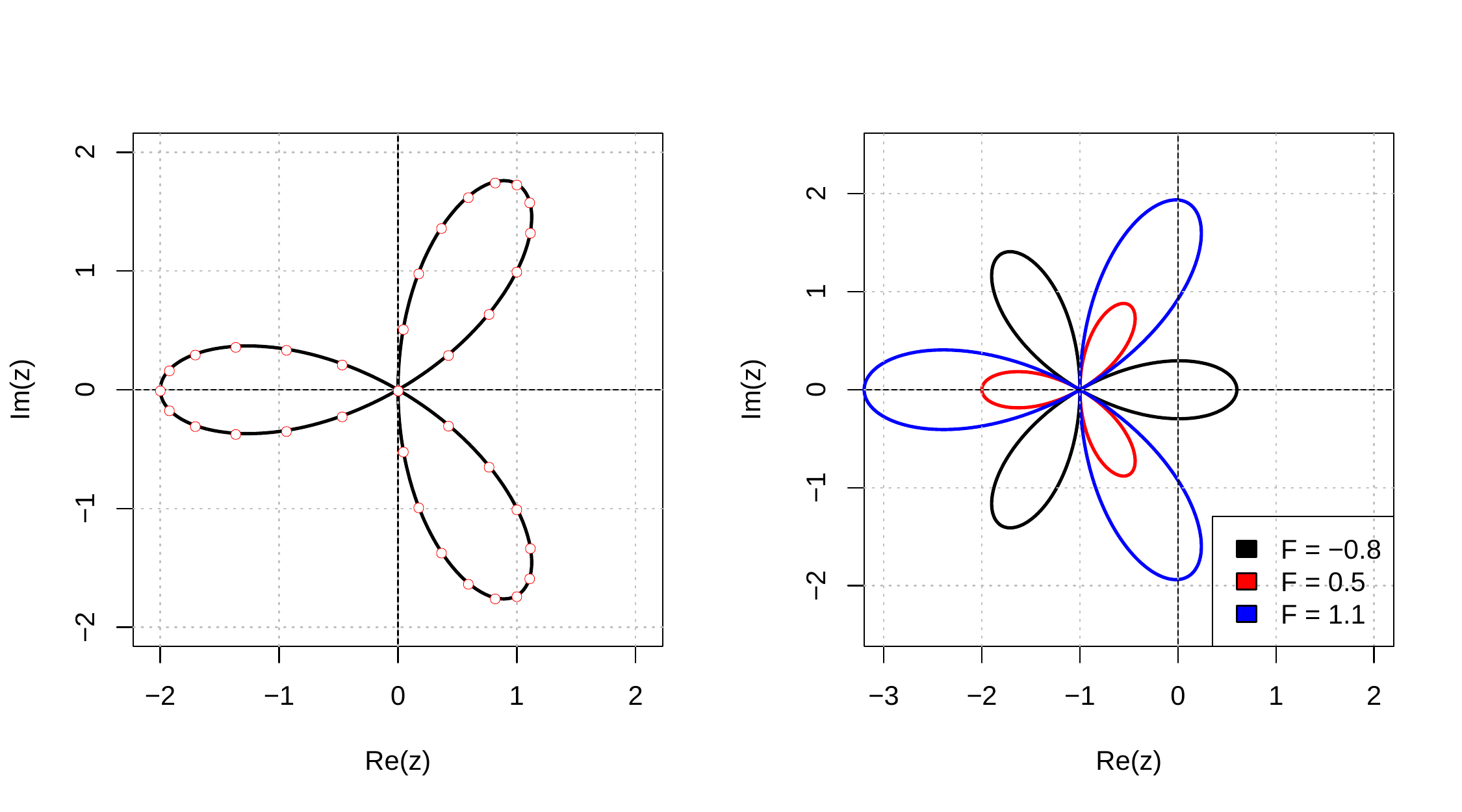}
  \caption{\textbf{Left:} Eigenvalue curve of the linearization of the L96 term $G_3$ about the constant vector $\mathbb{e}$ and location of eigenvalues on this curve for $N = 36$ sites. \textbf{Right:} Eigenvalue curves of $FA - I$ for the L96 system. The curve in the left panel is  stretched by the factor $F$ and then shifted to the left by 1 unit.  Black: $F = -0.8$. If $N$ is even, a pitchfork bifurcation has occurred for $F = -0.5$. Red: $F = 0.5$. The constant solution is stable. Blue: $F = 1.1$. A Hopf bifurcation has occurred for $F \approx 8/9$.  }
  \label{f-curves-1}
\end{figure}

Then the eigenvalues of $FA - I$ lie on a shifted version of such a curve, namely the image of $\mathbb{S}^1$ under the map $z \mapsto Fp_A(z) - 1$.  
This is shown in the right panel of Fig.~\ref{f-curves-1} for the case of the L96 system.  
For small $|F|$, the curve is always entirely in the left half plane and close to $-1$. If the constant term of $p_A$ vanishes, i.e. the diagonal of $A$ is zero, then also $p_A(1) = 0$ and this curve passes through the point $-1$ for all real $F$. As $F$ increases in magnitude, the image of the complex unit circle under $p_A$ is expanded. Then some points on the curve may cross the imaginary axis and reach into the right half plane, leading to bifurcations of the constant solution $\mathbf{x}_F$.

For the L96 case $G = G_3$ with $N= 36$ and $p_A(z) = - z^{-2} + z$, a real eigenvalue of $FA-I$ becomes positive as $F$ decreases below $- \frac{1}{2}$, resulting in a pitchfork bifurcation. 
For positive $F$,
a pair of complex eigenvalues crosses the imaginary axis as $F$ increases beyond $\left(\cos \frac{\pi}{9} + \sin \frac{\pi}{18} \right)^{-1} \approx 0.898198 $, another pair of eigenvalues  crosses the imaginary axis as $F$ increases beyond $\left( \cos \frac{2\pi}{9} + \sin \frac{\pi}{9} \right)^{-1} \approx 0.902474$, and two more pairs of eigenvalues  cross the imaginary axis as $F$ further increases beyond $1$, as the above discussion of the spectrum of the matrix $A$ implies.

Whenever a pair of complex eigenvalues crosses the imaginary axis due to increased forcing $F>0$, a Hopf bifurcation is expected to occur off the branch of constant solutions. As Fig.~\ref{f-curves-0} shows, such Hopf bifurcations are therefore also expected to occur for $G_5$ and $G_6$. The penultimate column in Table~1 lists the types of the first expected bifurcation for $G_1,\dots,G_8$ for increasing $F>0$. Whenever a bifurcation occurs for a $\mathscr{G}$-map, it also occurs for its tilde conjugate, since the shape of the curve does not change.

For the purpose of determining the initial bifurcation when $F$ is negative, taking a $\mathscr{G}$-map $G_j$ with $F = F^- < 0$ is equivalent to taking $-G_j$ with $F = -F^-$. Therefore, to visualize the initial bifurcation behavior for negative $F$ we may simply flip the eigenvalue curves about the imaginary axis. The last column in Table~1 gives the types of the first expected bifurcation for $G_1,\dots,G_8$ for decreasing $F<0$.

Some of the results in the rest of this section are independent of the sign of $F$. For the case of the L96 system, we will focus on the case of positive $F$.

\subsection{Normal Form at the First Hopf Bifurcation}
\label{sec-3-lyapunov}

We now present an approach to find the first Lyapunov coefficient for any system of the form Eq.~(\ref{e-lorenz-general}) that undergoes a Hopf bifurcation as $F$ increases. We assume throughout that $G$ is quadratic and equivariant. For the L96 system, we show that the first Hopf bifurcation as $F$ increases is always supercritical, i.e. its first Lyapunov exponent is negative, for all site numbers $N \ge 4$. 

Fix $N \ge 4$, let $G$ be a $\mathscr{G}$-map, let $A$ be the linearization of $G$ about the constant vector $\mathbf{e}$, and assume  such that there is a smallest $F = F_1 > 0$ such that $FA - I$ has a pair of purely imaginary eigenvalues, say $\pm i \tau_0$.
Then there is $0 < k < N/2$ such that $i \tau_0 = F_1 p_A(\omega_N^k) - 1$, where $p_A$ is the Laurent polynomial associated with $A$ and $\omega_N = \exp(2 \pi i/N)$. Note that there is more than one choice of $k$ if two eigenvalues cross the imaginary axis simultaneously for the same $F_1$. We write $z_1 = \omega_N^k$, and note that 
\begin{equation}
F_1 = \frac{1}{\Re p_A(z_1) } 
\, , \quad \tau_0 = \frac{\Im p_A(z_1)}{\Re p_A(z_1) } \, .
\label{e-F1-omega0}
\end{equation}
To compute the first Lyapunov exponent for this Hopf bifurcation, we use the approach described in \cite{kuznetsov1999} and \cite{kuznetsov2013}. This requires the computation of normal form coefficients for the cubic approximation of the system that is satisfied by $\mathbf{y} = \mathbf{x} - F_1 \mathbf{e}$. 

Since the system is already quadratic, the quadratic approximation is exact and has the form
\begin{equation}
\dot{\mathbf{y}} = \mathcal{A} \mathbf{y} + \tfrac12 \mathcal{B}(\mathbf{y},\mathbf{y})
\label{e-system-quadratic}
\end{equation}
where $\mathcal{A} = F_1A  - I$ and $\mathcal{B}$ is as in Eq.~(\ref{e-bilinear}). 

Consider first the general case of a system~(\ref{e-system-quadratic}), where $\mathcal{A}$ is assumed to be a general circulant matrix with an eigenvalue pair on the imaginary axis. Let $p_\mathcal{A}$ be the Laurent polynomial that is associated with $\mathcal{A}$. As summarized in Appendix A, the eigenvectors of $\mathcal{A}$ are the columns $\mathbf{q}_\ell$ of the normalized $N \times N$ Fourier matrix, and the eigenvalues are $p_\mathcal{A}(\omega_N^\ell)$. Let $\mathbf{q}_k$ be the right eigenvector of $\mathcal{A}$ for the eigenvalue $i \tau_0$, let $\bar{\mathbf{q}}_k = \mathbf{q}_{N-k}$ be its complex conjugate, and let $\mathbf{p}_k = \mathbf{q}_{k}$ be the eigenvector of $\mathcal{A}^T$ for the eigenvalue $- i \tau_0 = p_\mathcal{A}(\omega_N^k)$. All eigenvectors are normalized to unit length. According to \cite{kuznetsov1999} and \cite{kuznetsov2013}, section 5.4, the first Lyapunov coefficient then is    
\begin{equation}
I_1 = \frac{1}{2 \tau_0} \Re \left( 
- 2 \langle \mathbf{p}_k, \mathcal{B}(\mathbf{q}_k, 
\mathcal{A}^{-1} \mathcal{B}(\mathbf{q}_k, \bar{\mathbf{q}}_k)) \rangle + 
\langle \mathbf{p}_k, \mathcal{B}(\bar{\mathbf{q}}_k, (2 i \tau_0 I -  \mathcal{A})^{-1} \mathcal{B}(\mathbf{q}_k, \mathbf{q}_k)) \rangle 
\right)  
\label{e-Lyapunov-1}
\end{equation}
where $\langle\; , \, \rangle$ is the usual scalar product in $\mathbb{C}^N$ (linear in the second argument, conjugate-linear in the first). This quantity can be now computed in closed form, because the eigenvectors of a circulant matrix are columns of the Fourier matrix and $\mathcal{B}$ maps pairs of such vectors to multiples of these eigenvectors (by adding their indices).

\begin{lemma}
Consider system~(\ref{e-system-quadratic}) where $\mathcal{A}$ is a real circulant matrix with associated Laurent polynomial $p_\mathcal{A}$ and $\mathcal{B}$ comes from Eq.~(\ref{e-quadratic}) with equivariant $G$. Define $P_\mathcal{B}$ as in Eq.~(\ref{e-BPolynomial}). Assume that $\mathcal{A}$ has a pair of purely imaginary eigenvalue $i \tau_0 = p_\mathcal{A}(\omega_N^k), \, - i \tau_0$, and set $z_1 = \omega_N^k$.  Then the first Lyapunov exponent of the system~(\ref{e-system-quadratic}) is
\begin{equation}
I_1 = \frac{1}{2 \tau_0 N} \left( -2 P_\mathcal{B}(z_1, \bar{z}_1) \Re \frac{P_\mathcal{B}(z_1,1)}{p_\mathcal{A}(1)}  + \Re \frac{P_\mathcal{B}(z_1,z_1)P_\mathcal{B}(z_1^2,\bar{z_1})}{2 i \tau_0 - p_\mathcal{A}(z_1^2)} \right) \, .
\label{e-Lyapunov-2}
\end{equation}
\label{l-lemma1}
\end{lemma}

\begin{proof}
We obtain for the first term inside the parentheses in Eq.~(\ref{e-Lyapunov-1})
\begin{align*}
\mathcal{B}(\mathbf{q}_k, \bar{\mathbf{q}}_k) &= \frac{P_\mathcal{B}(z_1, \bar{z_1})}{\sqrt{N}} \mathbf{q}_0= \frac{P_\mathcal{B}(z_1, \bar{z_1})}{\sqrt{N}} \mathbf{e}\\
\mathcal{A}^{-1} \mathcal{B}(\mathbf{q}_k, \bar{\mathbf{q}}_k)&=  \frac{1}{p_\mathcal{A}(1)}\frac{P_\mathcal{B}(z_1, \bar{z_1})}{\sqrt{N}} \mathbf{e}\\
\mathcal{B}(\mathbf{q}_k, \mathcal{A}^{-1} \mathcal{B}(\mathbf{q}_k, \bar{\mathbf{q}}_k))&=
 \frac{P_\mathcal{B}(z_1, \bar{z_1})P_\mathcal{B}(z_1, 1)}{p_\mathcal{A}(1)N} \mathbf{q}_k\\
\langle \mathbf{q}_k, \mathcal{B}(\mathbf{q}_k, \mathcal{A}^{-1} \mathcal{B}(\mathbf{q}_k, \bar{\mathbf{q}}_k))\rangle &=
 \frac{P_\mathcal{B}(z_1, \bar{z_1})P_\mathcal{B}(z_1, 1)}{p_\mathcal{A}(1)N} \, .
\end{align*}
Note that $P_\mathcal{B}(z_1, \bar{z_1}) \in \mathbb{R}$.  For the second term we obtain similarly
\begin{align*}
\mathcal{B}(\mathbf{q}_k, \mathbf{q}_k)&=\frac{P_\mathcal{B}(z_1,z_1)}{\sqrt{N}} \mathbf{q}_{2k} \\
(2 i \tau_0 I -  \mathcal{A})^{-1} \mathcal{B}(\mathbf{q}_k, \mathbf{q}_k) &=
\frac{1}{2 i \tau_0 - p_\mathcal{A}(z_1^2) } \frac{P_\mathcal{B}(z_1,z_1)}{\sqrt{N}} \mathbf{q}_{2k}\\
\mathcal{B}(\bar{\mathbf{q}}_k, (2 i \tau_0 I -  \mathcal{A})^{-1} \mathcal{B}(\mathbf{q}_k, \mathbf{q}_k))
&= \frac{1}{2 i \tau_0 - p_\mathcal{A}(z_1^2)} \frac{P_\mathcal{B}(z_1,z_1)P_\mathcal{B}(z_1^2,\bar{z}_1)}{N} \mathbf{q}_{k}\\
\langle \mathbf{p}_k, \mathcal{B}(\bar{\mathbf{q}}_k, (2 i \tau_0 I -  \mathcal{A})^{-1} \mathcal{B}(\mathbf{q}_k, \mathbf{q}_k)) \rangle &=  \frac{P_\mathcal{B}(z_1,z_1)P_\mathcal{B}(z_1^2,\bar{z}_1)}{(2 i \tau_0  - p_\mathcal{A}(z_1^2))N} \, .
\end{align*}
Combining these two computations, we obtain Eq.~(\ref{e-Lyapunov-2}).
\end{proof}

For the L96 system, it is now possible to show that the first Hopf bifurcation is always supercritical, for all site numbers $N \ge 4$.

\begin{theorem}
Consider the L96 system for $N \ge 4$ sites. Let $F_1$ be the smallest positive forcing for which the linearization $FA - I$ has eigenvalues on the imaginary axis. Then the first Lyapunov exponent is negative, i.e. the resulting Hopf bifurcation is supercritical. 
\label{t-thm2}
\end{theorem}

\begin{proof}
We start by recalling that $\mathcal{A} = F_1 A - I$ and therefore $p_\mathcal{A}(z) = F_1 p_A(z) - 1, \, p_A(z) = z - z^{-2}$. Since $p_A(1) = 0$, therefore $p_\mathcal{A}(1) = -1$. Next, recall that in the L96 case  $P_\mathcal{B}(z,w) = P_L(z,w) = (z-z^{-2})w^{-1} + (w-w^{-2})z^{-1}$, see Eq.~(\ref{e-LPolynomial}).  

Now examine the terms in Eq.~(\ref{e-Lyapunov-2}). Let $t_1 = 2 k \pi/N$ and $z_1 = \exp(i t_1)$ as before such that $F_1 p_A(z_1) - 1 = i \tau_0$. Then since $P_L(z_1,\bar{z}_1) = 2 \cos 2 t_1 - \cos t_1$ and  $P_L(z,1) = z - z^{-2} = p_A(z)$, we obtain $\Re p_A(z_1) = \cos t_1 - \cos 2t_1$. The first term in the parentheses in  Eq.~(\ref{e-Lyapunov-2}) therefore becomes
$- 4 \left(\cos t_1 - \cos 2 t_1 \right)^2$.

For the second term in the parentheses, we first observe that
$P_L(z_1,z_1)P_L(z_1^2,\bar{z}_1) = 2(z_1^3 - 2 + z_1^{-3}) =  4 (\cos 3 t_1 - 1)$ is also a real number. Therefore Eq.~(\ref{e-Lyapunov-2}) now takes the form
\begin{equation}
I_1 = \frac{4}{2 \tau_0 N} \left(-\left(\cos t_1 - \cos 2 t_1 \right)^2 + (\cos 3 t_1 - 1)\Re \, \frac{1}{2 i \tau_0 +1 - F_1p_A(z_1^2)} \right) .  
\label{e-Lyapunov-3}
\end{equation}
Here $\frac{2 \pi}{7} \le t_1 \le \frac{\pi}{2}$, where the lower bound is attained for $N = 7$ and the upper bound for $N = 12$. Now $-\left(\cos t - \cos 2 t \right)^2$ and $\cos 3t - 1$ are both strictly negative on $(0, 2 \pi/3)$. Also clearly $F_1 p_A(z_1^2) - 1$ cannot have positive real part, since this number is an eigenvalue of $F_1 A - I$ and none of these eigenvalues are in the right half plane. Consequently, $\Re  \, \frac{1}{2 i \tau_0 +1 - F_1p_A(z_1^2)}$ is non-negative. All this implies that $I_1 < 0$ for all possible $t_1$ and therefore the first Hopf bifurcation is always supercritical. 
\end{proof}

We note that the same approach may be used to show that the first Hopf bifurcations for the advection terms $G_5$ and $G_6$, defined in Eq.~(\ref{e-form-k3}), are also always supercritical. 

\subsection{Determining Advection Terms from the Linearization}
\label{sec-3-adv-from-lin}

We now discuss whether it is possible to retrieve advection terms given just the eigenvalue curve. Consider the spaces $\mathscr{G}_1,\, \mathscr{G}_2,\, \mathscr{G}_3$ identified in Theorem~\ref{thm-all.quadratic}. Linearizing a $\mathscr{G}$-map about the constant vector $\mathbf{e}$ results in a circulant matrix $A$ whose top row gives the coefficients of an associated Laurent polynomial $p_A$. This is clearly a linear operation. A $k$-localized $\mathscr{G}$-map has a Laurent polynomial of the form $d_{-k}z^{-k} + \dots + d_k z^k$ whose coefficients $d_j = A_{0j} = \sum_\ell Q_{j\ell}$ must add to zero, as discussed earlier.

Observe first that for maps in $\mathscr{G}_1$ this defines a bijection between two two-dimensional vector spaces. Thus for any Laurent polynomial $p_A(z) = d_{-1}z^{-1} + d_0 + d_1z$ with $d_{-1} + d_0 + d_1 = 0$ there exists a unique 1-localized $\mathscr{G}$-map.

Consider next the set of maps in $\mathscr{G}_2$ and the associated Laurent polynomials, which have the form 
\begin{equation}
p_A(z) = d_{-2}z^{-2} + d_{-1}z^{-1} + d_0 + d_1z + d_2z^2 .
\label{e-Laurent-2}
\end{equation}
The space $\mathscr{G}_2$ is six-dimensional, while the associated Laurent polynomials belong to a four-dimensional space, due to the condition $\sum_{|j| \le 2} d_j  =0$. Indeed, it is easy to verify that any four of the Laurent polynomials of the form Eq.~(\ref{e-Laurent-2}) in Table~1 along with their tilde conjugate polynomials $p_A(z^{-1})$ span this space. Consequently there is a two-dimensional null space of such maps whose associated Laurent polynomial vanishes. One can verify that this null space is spanned by the map
\begin{equation}
G_0 = G_3 - 2 \tilde G_3 + \tilde G_1 - G_2 
\label{e-G0}
\end{equation}
together with the associated $\tilde G_0$. A Laurent polynomial of the form Eq.~(\ref{e-Laurent-2}) determines a 2-localized $\mathscr{G}$-map up to multiples of $G_0$ and $\tilde G_0$. In particular, for systems of the form Eq.~(\ref{e-lorenz-general}) where $G$ is a linear combination of $G_0$ and $\tilde G_0$, the linearization about any stationary solution $\mathbf{x}_F = F \mathbf{e}$ always has the form $\dot{\mathbf{y}} = - \mathbf{y}$. Such stationary solution are therefore globally asymptotically stable for the full system and must be unique.

The space $\mathscr{G}_3$ is 12-dimensional, while the space of associated Laurent polynomials is 6-dimensional, resulting in a 6-dimensional null space of 3-localized $\mathscr{G}$-maps whose linearizations about $\mathbf{e}$ vanish.

\subsection{Eigenvalue Curves Resulting in a Hopf Bifurcation}
\label{sec-3-Hopf}

We wish to identify conditions such that the resulting complex curve has a pair of lobes, symmetric about the real axis, on the right, which will therefore cross the imaginary axis simultaneously and with positive velocity as 
$F$ increases. As a pair of eigenvalues on these lobes crosses the imaginary axis, the constant stationary solution $F \mathbf{e}$ will lose its stability in a Hopf bifurcation.

We will first consider the simplest case of a 2-localized eigenvalue curve. Take $G \in \mathscr{G}_2$ and let $p_A$ be the associated Laurent polynomial of the linearization of $G$ about $\mathbf{e}$. Given $F \in \mathbb{R}$, the real and imaginary parts of the eigenvalues of $FA -I$ are of the form $F p_A(e^{2\pi i s}) - 1$ where $p_A$ is of the form Eq.~(\ref{e-Laurent-2}). The real and imaginary parts of $p_A(e^{2\pi i s})$ can be written as
\begin{equation}
\begin{aligned}
    &\Lambda_R(s) := \Re p_A(e^{2\pi i s}) = R_0 + R_1 \cos (2\pi s) + R_2 \cos (4\pi s) \\
    &\Lambda_I(s) := \Im p_A(e^{2\pi i s}) = I_1 \sin (2\pi s) + I_2 \sin (4\pi s).
\end{aligned}
\label{e-eig-curve-re-im}
\end{equation}
Here $R_j = d_j + d_{-j}, \, I_j = d_j - d_{-j}$ for $j = 1, \, 2$ and $R_0 = -R_1 - R_2$. 

Note first that $\Lambda_R$ is even about $s = 1/2$ while $\Lambda_I$ is odd about this point. A calculus argument shows that $\Lambda_R$ has critical points at $s=0$ and $s=1/2$. If $R_2 \ne 0$ and $-1 < \frac{R_1}{4R_2} < 1$, there are two additional critical points  
\begin{equation}
s_1 = \frac{1}{2\pi} \cos^{-1}  \big\vert \frac{R_1}{4R_2} \big\vert \, ,  \quad s_2 = 1 - s_1 .
\label{e-2s-values}
\end{equation}
These are positive maxima if $4R_2 < R_1 < -4R_2$, i.e. $R_2<0$, and negative minima otherwise. Therefore in the case $R_2<0$ a Hopf bifurcation will occur for $F>0$, while in the case $R_2>0$ a Hopf bifurcation will occur for $F<0$.

Since $\Lambda_I(s)$ is odd about $s =1/2$, it follows that $\Lambda_I(s_1) = - \Lambda_I(s_2)$. Requiring that these two values are distinct is equivalent to the constraint $I_2R_1 + 2I_1R_2 \neq 0$. 

In summary, in order for a L96-like system with a 2-localized advection term to undergo a Hopf bifurcation for some value of the forcing parameter $F$, it is sufficient that its Laurent polynomial $p_A$ satisfy the two following conditions:
\begin{enumerate}
    \item $\vert d_{-1}+d_1 \vert < 4 \vert d_{-2}+d_2 \vert$
    \item $(d_{-2}-d_2)(d_{-1}+d_1) \neq 2(d_1-d_{-1})(d_{-2}+d_2)$
\end{enumerate}
Then, for large $N$, the first Hopf bifurcation will occur very near $F = F_1 = \Lambda_R (s_1) ^{-1}$.

There is another case which can lead to a Hopf bifurcation. Note that when $N$ is even, the critical point $s = 1/2$ of $\Lambda_R$ corresponds always to a real eigenvalue of $A$, and then the stationary solution $F \mathbf{e}$ may lose its stability  in a pitchfork bifurcation. But if $N$ is odd, there is no eigenvalue corresponding to $s = 1/2$ and there are two complex conjugate eigenvalues at $s = \frac{1}{2} \pm \frac{1}{2N}$. These may cross the imaginary axis as the magnitude of $F$ increases, resulting in a Hopf bifurcation. If $s = 0$ and $s = 1/2$ are the only critical points of $\Lambda_R$, i.e. $\vert R_1 \vert \geq 4 \vert R_2 \vert$, this always happens. If the critical points of $\Lambda_R$ in Eq.~(\ref{e-2s-values}) exist, i.e. $\vert R_1 \vert < 4\vert R_2 \vert$, then this may happen for sufficiently large $N$. In the remaining cases, $\Lambda_R$ is non-positive everywhere and no bifurcation off the constant stationary solution is possible. Details are left to the reader.

For $k$-localized advection terms with $k \geq 3$, Eq.~(\ref{e-eig-curve-re-im}) may be extended in the obvious way. The problem of finding the nontrivial critical points of $\Lambda_R$ may always be reduced to finding the roots of a $(k-1)$-th degree polynomial in $\cos(2\pi s)$. For each real root $r$ satisfying $-1 < r < 1$, we get a pair solutions of the form Eq.~(\ref{e-2s-values}) with $R_1/4R_2$ replaced by $r$. These are symmetric about $s=1/2$ with $0 < s_1 < 1/4$ and thus $\vert s_2 - s_1 \vert > 1/2$.

In Figure~\ref{fig-hopf-g1-g5} are plotted some examples of eigenvalue curves, multiplied by $F$ and shifted by one to the left. The eigenvalue curve in Fig.~\ref{fig-hopf-g1-g5}c results from a $\mathscr{G}$-map that is a linear combination of the $\mathscr{G}$-maps that produce the eigenvalue curves in Fig.~\ref{fig-hopf-g1-g5}a and Fig.~\ref{fig-hopf-g1-g5}b. A Hopf bifurcation occurs for positive $F$ in Fig.~\ref{fig-hopf-g1-g5}b and Fig.~\ref{fig-hopf-g1-g5}c, but no bifurcation occurs for positive $F$ in Fig.~\ref{fig-hopf-g1-g5}a.

\begin{figure}[ht]
    \centering
        \includegraphics[width=0.32\textwidth]{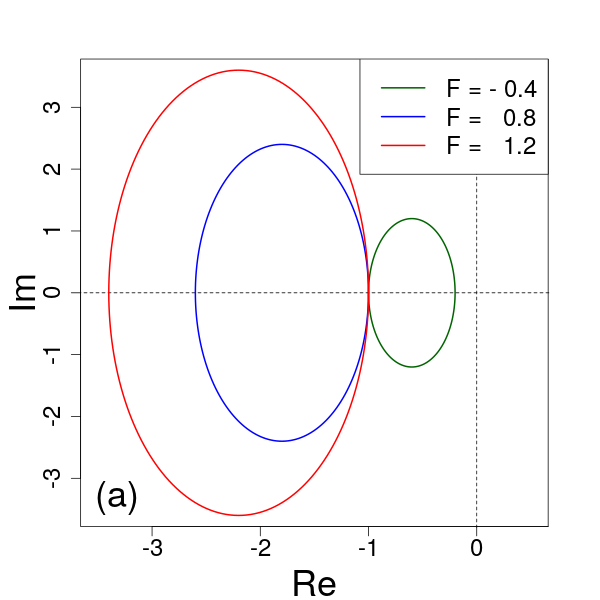}
        \includegraphics[width=0.32\textwidth]{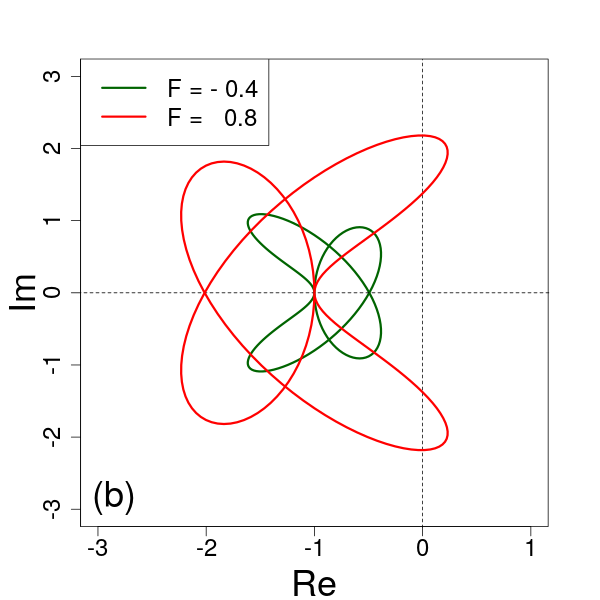}
        \includegraphics[width=0.32\textwidth]{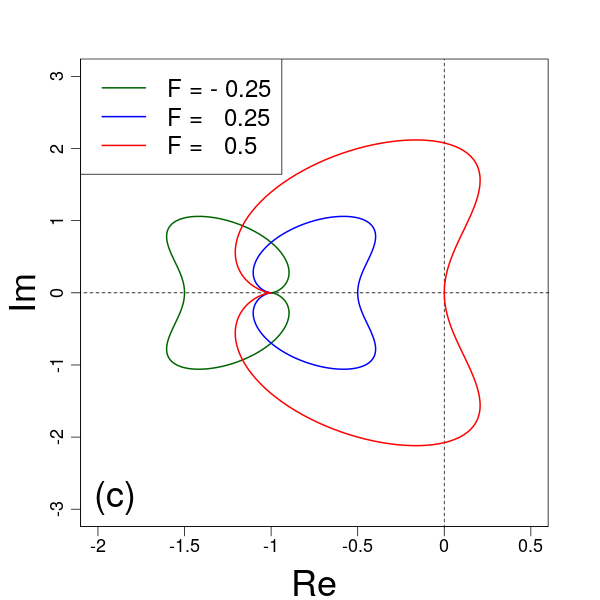}
        \caption{Eigenvalue curves of advection terms in $\mathscr{G}_2$ for various values of $F$. (a) $G_1$, no bifurcation occurs for positive $F$. (b) $G_5$, a supercritical Hopf bifurcation occurs for positive $F$. (c) $-G_1 + \frac{1}{2} G_5$, a supercritical Hopf bifurcation occurs for positive $F$.}
        \label{fig-hopf-g1-g5}
\end{figure}

\subsection{Waves Resulting from the Hopf Bifurcation}
\label{sec-3-waves}

In this section we describe some properties of the stable limit cycle following the first Hopf bifurcation for L96-like systems.

Recall that the $j$-th eigenvector-eigenvalue pair of a $N \times N$ circulant matrix with elements $A_{mn}$ is (see also Eq.~(\ref{e-FourierMatrix}))
\begin{equation}
    v_j = (v_j^k)_{0 \le k < N} = (e^{2\pi i j k / N})_{0 \le k < N}, \ \lambda_j = \sum_{\ell} A_{0\ell} e^{2\pi i j\ell/N}.
\label{e-Fourier1}    
\end{equation}
Therefore small perturbations $\mathbf{y}(t) = (y_0(t),\dots,y_{N-1}(t))^\text{T}$ from the stationary solution of Eq.~(\ref{e-lorenz-general}) with linearization  Eq.~(\ref{e-lorenz-2lin}) evolve over a short time approximately as
$$
y_k(t) = \sum_j b_j v_j^k e^{(F \lambda_j - 1)t} = \sum_j b_j e^{F(\Re\lambda_j -1 ) t} e^{2\pi i j k/N + i F \Im\lambda_j t}
$$
for some coefficients $b_j$ determined by the initial perturbation.

Assume $N$ is large so that $j/N$ can be approximated by $s \in [0,1)$ and suppose $F$ is slightly greater than its first Hopf bifurcation value so that only one pair of eigenvalues has positive real part. Then the $s_1$ mode will travel with some amplitude $\vert W \vert$ according to
\begin{equation}
y_k(t) = W e^{i (2\pi s_1 k + F\Lambda_I(s_1) t ) }.
\label{e-s1-mode}
\end{equation}
where $\Lambda_I$ is as in Eq.~(\ref{e-eig-curve-re-im}). This mode has wavelength $s_1^{-1} > 4$ sites since $0 < s_1 < 1/4$. The $s_2$ mode gives the same result (up to an irrelevant overall sign in the exponent) since $k$ is an integer, and therefore $\exp(2\pi i (1-s_1) k) = \exp(-2\pi i s_1 k)$.

As $F$ increases, other pairs of eigenvalues will cross the imaginary axis, so more modes will grow. Since the $s_1$ mode grows fastest, it is expected to dominate. We can approximate the phase velocities and group velocity of the resulting wave as
\begin{equation}
c_p(s) = -\frac{F\Lambda_I(s)}{2\pi s} \, , \quad c_g = \frac{F\Lambda^\prime_I (s_1)}{2\pi}
\label{e-wave-vel}
\end{equation}
where the prime denotes a derivative with respect to $s$. Note that $\text{sgn}(c_p(s))$ is the same for all $s$ due to the symmetry of the eigenvalue curve. If we replace $G$ by $\tilde{G}$ in our system, then the sign of the phase velocity is reversed since the eigenvalue curve is traced in the opposite direction.

For 2-localized advection terms, some algebra shows that $c_p$ and $c_g$ have the same sign if and only if the coefficients $I_1, \, I_2$ in Eq.~(\ref{e-eig-curve-re-im}) satisfy
$\frac{I_2}{I_1+2s_1I_2}>\frac{s_1}{2(1-s_1^2)}$ and opposite sign if the inequality is reversed. The reverse inequality always holds if $I_2$ is negative. For the L96 system, the phase velocity is negative and the group velocity is positive with roughly the same magnitude, so disturbances travel in the opposite direction as the wave train, as can be seen in Fig.~\ref{fig-hov-f8}.

\section{Coexistence of Stable Limit Cycles}
\label{sec-4}

Consider again the L96 system on $N$ sites. As $F$ increases past the first Hopf bifurcation value $F_1$, a stable limit cycle always appears.  A second unstable limit cycle appears at the second Hopf bifurcation point $F_2 > F_1$. There is abundant numerical evidence that as $F$ increases further past some value $F_3 > F_2$,  there exist two stable limit cycles.  In fact, for $N = 12$, two pairs of complex eigenvalues cross the imaginary axis simultaneously (a Hopf-Hopf bifurcation) as $F$ increases past $F_1 = F_2 = 1$, and two stable limit cycles appear immediately. 

An explanation for this phenomenon was given in \cite{VanKekem2018travelling}. The authors construct an embedding of the L96 system~(\ref{e-lorenz-1}) into a two-parameter system by adding a linear term $\alpha C \mathbf{x}$ to the right hand side of Eq.~(\ref{e-lorenz-1}), where $C$ is a suitably chosen circulant matrix. For some small positive $\alpha_0$ that depends on $N$, as $F$ increases, two pairs of eigenvalues then cross the imaginary axis simultaneously at some $F = \tilde F$.

By the analysis in \cite{kuznetsov2013}, under suitable conditions there exist two stable limit cycles for $ F > \tilde F, \, \alpha = \alpha_0$. These conditions may be expressed in terms of normal form coefficients for the system. Then if $\alpha_0$ is sufficiently small, it is plausible that this occurs also for $\alpha$ near $\alpha_0$, e.g. for $\alpha = 0$. In \cite{VanKekem2018travelling}, this is carried out largely numerically.

Our contribution to this question consists in the following modifications and extensions of the approach in \cite{VanKekem2018travelling}.

\begin{enumerate}
\item
We use a new choice of perturbation matrix $C$, which simplifies the analysis.
\item
We show how to compute all normal form coefficients analytically, using Proposition~\ref{prop-Fourier} in Appendix A.
\item 
We approximate $F_3$ using normal form coefficients and compare this approximation to numerical experiments.
\end{enumerate}

\subsection{Perturbation Near Hopf-Hopf Bifurcation}

For $0 \le j < N$, let $\lambda_j(F) = F p_A(\omega_N^j) - 1$ be the $j$-th eigenvalue of the linearization $FA - I$ about the constant state $F \mathbf{e}$.    
Let $\lambda_k(F), \, \lambda_{N-k}(F)$  be the first pair of complex eigenvalues that crosses the imaginary axis at $F = F_1$, i.e. $\lambda_k(F_1) = i \tau_1, \, \lambda_{N-k}(F_1) = -i \tau_1$. Similarly let $\lambda_\ell(F), \, \lambda_{N-\ell}(F)$ be the second pair that crosses the imaginary axis at $F = F_2 > F_1$, with $\lambda_\ell(F_2) = i \tau_2, \, \lambda_{N-\ell}(F_2) = -i \tau_2$. Typically, $|k - \ell| = 1$. Consider now the Fourier matrix $\mathcal{F}_N$, defined in Eq.~(\ref{e-FourierMatrix}), with columns $\mathbf{q}_j, \, 0 \le j < N$. Define the matrix $C_\ell = \mathbf{q}_\ell \mathbf{q}_{N-\ell}^{\text{T}} + \mathbf{q}_{N-\ell} \mathbf{q}_\ell^{\text{T}}$. This is a real valued rank 2 circulant matrix with eigenvectors $\mathbf{q}_\ell$ and $\mathbf{q}_{N-\ell}$, both with eigenvalue 1. All other columns $\mathbf{q}_j$ are in the null space of $C_\ell$.        

Now consider the perturbed system
\begin{equation}
\dot{\mathbf{x}} = G_L(\mathbf{x}) - \mathbf{x} + \alpha C_\ell\mathbf{x}  + F \mathbf{e}
\label{e-lorenz-1v-pert}
\end{equation}
for $\alpha \ge 0$. Its linearization about the constant state $F \mathbf{e}$  is $\dot{\mathbf{y}} = (FA - I + \alpha C_\ell)  \mathbf{y}$ and its eigenvalues are $\tilde \lambda_j(F) = Fp_A(\omega_N^j) - 1+ \alpha (\delta_{j\ell} + \delta_{j, N-\ell})$, where $\delta_{j\ell}$ is the Kronecker delta. Thus for any fixed $\alpha \ge 0$, $(\lambda_k(F), \, \lambda_{N-k}(F))$ is still a pair of eigenvalues of the perturbed linearization that crosses the imaginary axis at $F = F_1$.  The pair of eigenvalues $(\tilde \lambda_\ell(F), \, \tilde \lambda_{N-\ell}(F)) = (\lambda_\ell(F) + \alpha, \, \lambda_{N- \ell}(F) + \alpha)$ of $FA - I + \alpha C_{\ell}$  crosses the imaginary axis at $F = F_2(1-\alpha)$. For $\alpha = \alpha_0 : = \frac{F_2 - F_1}{F_2}$, both pairs of eigenvalues cross the imaginary axis simultaneously at $F = F_1$.

This is illustrated in Fig.~\ref{f-unfolding}. The left panel shows the eigenvalues of $F_1A - I$ (black and red circles) and the eigenvalues of $F_1A - I + \alpha_0 C_\ell$ (black and green circles), for $N=14$ sites. The right panel shows curves in the $(F, \alpha)$ plane where bifurcations occur.  The values $F = F_1$ (blue circle) and $F = F_2$ (green circle) are marked on the axis $\{\alpha = 0 \}$ (black). There is a Hopf bifurcation for system~(\ref{e-lorenz-1v-pert}) for all $\alpha \ge 0$ at $F = F_1$ (blue line), when $\lambda_k(F)$ and its complex conjugate cross the imaginary axis. There is also a Hopf bifurcation for all $F = F_2(1-\alpha)$ (green line), when $\tilde \lambda_\ell(F)$ and its complex conjugate cross the imaginary axis. For $F = F_1, \, \alpha = \alpha_0$ (red triangle), two pairs of eigenvalues are on the imaginary axis. The remaining features in the right panel will be explained below.

We note that all Hopf bifurcations described here are supercritical. Their first Lyapunov coefficients can be expressed in ways that are completely analogous to Eq.~(\ref{e-Lyapunov-3}), with $i \tau_0$ replaced by some other imaginary number and $F_1$ replaced by $F = F_2(1 - \alpha)$. We leave the details to the reader.    

\begin{figure}[ht]
    \centering
        \includegraphics[width = 0.95\textwidth]{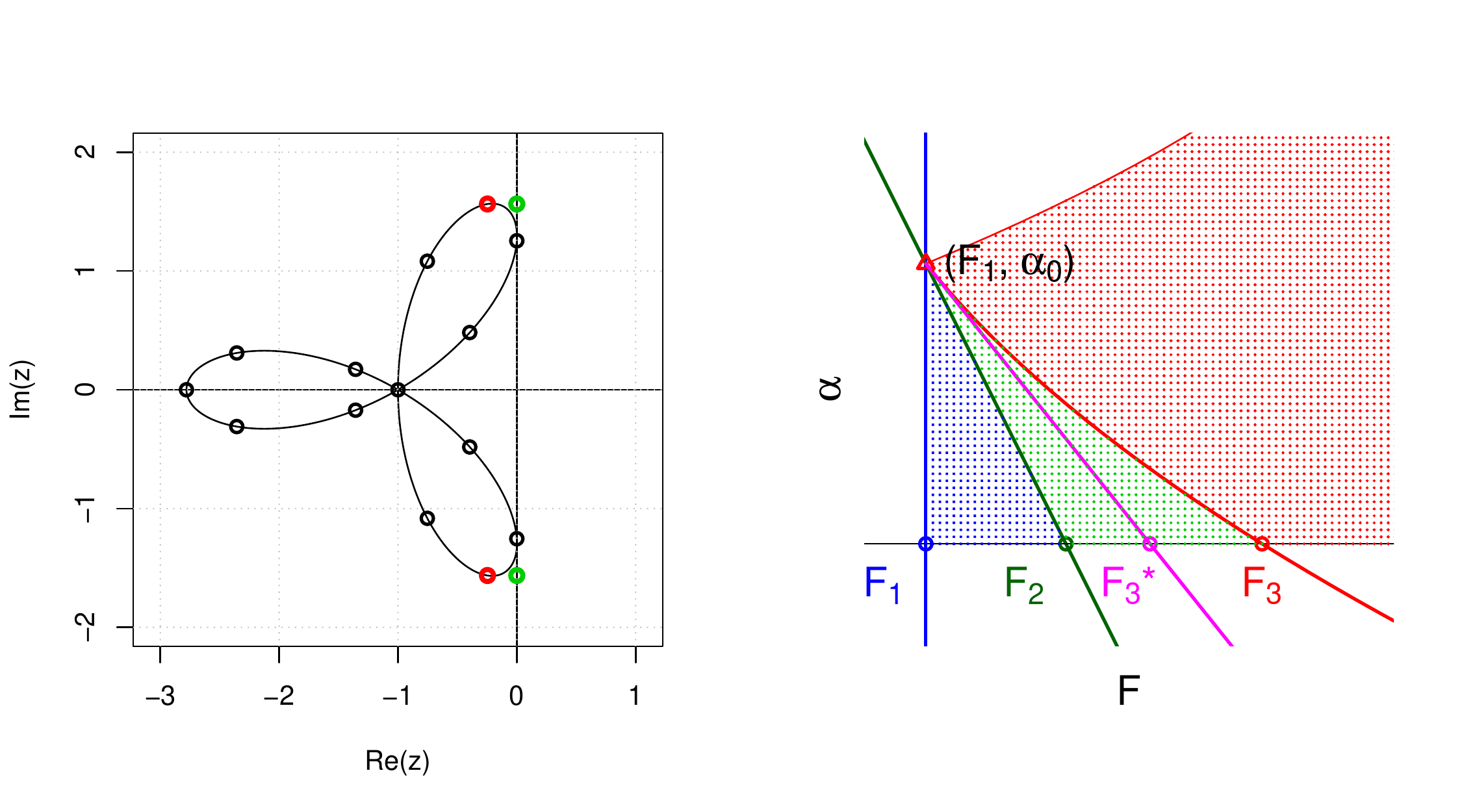}
        \caption{\textbf{Left:} Eigenvalues of $F_1A - I$ (black and red) and of $F_1A - I + \alpha_0 C_\ell$ (black and green), for $N=14$. 
        \textbf{Right:} Bifurcation diagram of system~(\ref{e-lorenz-1v-pert}) in the $(F,\alpha)$ plane. Blue line: Hopf bifurcation ($\Re \, \lambda_k = 0$).  Blue stipples: A stable limit cycle exists. Green line: Hopf bifurcation ($\Re \, \tilde \lambda_\ell = 0$). Green stipples: A second periodic orbit exists (unstable). Red curves: Neimark-Sacker (N-S) bifurcation. Red stipples: Two stable limit cycles coexist. Magenta line: Linear approximation of N-S bifurcation curve. Red triangle: Hopf-Hopf bifurcation at $(F_1, \, \alpha_0)$.
        }
\label{f-unfolding}
\end{figure}

\subsection{Normal Form at Hopf-Hopf Bifurcation}
\label{sec-4-normal-form}

We now examine the perturbed system~(\ref{e-lorenz-1v-pert}) for $\alpha = \alpha_0, \, F = F_1$. At this set of parameters, two  eigenvalues $\lambda_k(F_1) = i \tau_1 = F_1 p_A(z_1) - 1$ and $\tilde \lambda_\ell(F_1)  = i \tau_2 = F_1 p_A(z_2) - 1 + \alpha_0$ together with their conjugates are on the imaginary axis. Here $z_1 = \omega_N^k$ and $z_2 = \omega_N^\ell$.

The corresponding right eigenvector pairs of the linearization are $\mathbf{q}_k, \, \mathbf{q}_{N-k}$ for the first pair of eigenvalues and $\mathbf{q}_\ell, \, \mathbf{q}_{N-\ell}$ for the second pair. The left eigenvectors of the linearization for all these eigenvalues are the conjugate transposes of the right eigenvectors, i.e. $\mathbf{q}_{N-k}, \, \mathbf{q}_k$ for the first pair and $\mathbf{q}_{N-\ell}, \, \mathbf{q}_\ell$ for the second pair.

To analyze the behavior of the full system near $\alpha = \alpha_0, \, F = F_1$, we use the approach described in \cite{kuznetsov1999} and \cite{kuznetsov2013}, section 8.6. This requires the computation of normal form coefficients for the cubic approximation of the system that is satisfied by $\mathbf{y} = \mathbf{x} - F_1 \mathbf{e}$. As before, this approximation is exact and has the form Eq.~(\ref{e-system-quadratic}), where $\mathcal{B} = \mathcal{B}_L$ is as in Eq.~(\ref{e-bilinear}) and $\mathcal{A} = F_1A + \alpha_0C_\ell - I$.

To study the behavior of the system near $F = F_1, \, \alpha = \alpha_0$, four normal coefficients $p_{rs}, \, r,s = 1,2$ must be computed. These all have the same general form as the first Lyapunov coefficient defined in Eq.~(\ref{e-Lyapunov-1}). For example, 
\begin{align*}
p_{12} =  \Re \, \big( &\langle \mathbf{q}_k, \mathcal{B}_L((i(\tau_1 + \tau_2)I - \mathcal{A})^{-1} \mathcal{B}_L(\mathbf{q}_k, \mathbf{q}_\ell), \mathbf{q}_{N-\ell} \rangle \\
 + &\langle \mathbf{q}_k, \mathcal{B}_L((i(\tau_1 - \tau_2)I - \mathcal{A})^{-1} \mathcal{B}_L(\mathbf{q}_k, \mathbf{q}_{N-\ell}), \mathbf{q}_\ell \rangle \\
  + &\langle \mathbf{q}_k, \mathcal{B}_L(- \mathcal{A}^{-1} \mathcal{B}_L(\mathbf{q}_\ell, \mathbf{q}_{N-\ell}), \mathbf{q}_k \rangle \big)
  \label{e-hopfhopfnormal}
\end{align*}
Since $\mathcal{B}_L$ maps pairs of eigenvectors of $\mathcal{A}$ to multiples of eigenvectors (see Proposition~\ref{prop-Fourier} in Appendix A), these terms can all be evaluated in closed form, in the same manner as in Lemma~\ref{l-lemma1} and Theorem~\ref{t-thm2}. For example, writing $z_1 = \omega_N^k, \, z_2 = \omega_N^\ell$, one obtains
\begin{equation}
p_{12} =  \frac{1}{N}\Re \, \left(\frac{P_L(z_1, z_2) P_L(z_1z_2, \bar z_2)}{i(\tau_1 + \tau_2) +1 - F_1p_A(z_1z_2)} 
+ \frac{P_L(z_1, \bar z_2) P_L(z_1\bar z_2, z_2)}{i(\tau_1 - \tau_2) +1 - F_1p_A(z_1 \bar z_2)} 
- P_L(z_2, \bar z_2) P_L(1, z_1) \right) 
\end{equation}
The other three normal form coefficients $p_{rs}$ can be expressed similarly in terms of the Laurent polynomials $p_A$ and $P_L$ and can therefore be evaluated numerically.  The details of this calculation are left to the reader.

\subsection{Neimark-Sacker Bifurcation}

Consider now the perturbed system in Eq.~(\ref{e-lorenz-1v-pert}) for $\alpha = \alpha_0 = 1 - F_1/F_2$. As $F$ increases past $F_1$, two eigenvalue pairs cross the imaginary axis simultaneously, resulting in a Hopf-Hopf bifurcation. The bifurcation picture now depends on the normal form coefficients $p_{ij}, \, i,\, j = 1, \, 2$ which may be computed as in the previous section. A complete description is given in \cite{kuznetsov2013}, section 8.6. 

In all cases of the L96 system that we examined, it turns out that $p_{11} p_{22} > 0$ (the ``simple case'' in the terminology of  \cite{kuznetsov2013}) and $p_{11}p_{22} - p_{12}p_{21} < 0$ (``type I'' in that reference). 
By the results in  \cite{kuznetsov2013}, there exist smooth curves $\alpha = \gamma_1(F)$ (thick red in Fig.~\ref{f-unfolding}) and $\alpha = \gamma_2(F)$ (thin red in Fig.~\ref{f-unfolding}), both passing through the point $(\alpha_0, F_1)$, such that for $F > F_1$ and close to $F_1$ and for $\gamma_1(F) < \alpha < \gamma_2(F)$ there exist two stable limit cycles for the system~(\ref{e-lorenz-1v-pert}) (red stipples in~Fig.~\ref{f-unfolding}). For   
$1 - F/F_2 < \alpha < \gamma_1(F)$ a stable limit cycle and an unstable limit cycle coexist (green stipples in~Fig.~\ref{f-unfolding}). At the transition curve $\alpha = \gamma_1(F)$, a subcritical Neimark-Sacker (N-S) bifurcation occurs in which a two-dimensional unstable torus bifurcates from the unstable limit cycle. As a result, this limit cycle now becomes stable. Recall that the unstable limit cycle comes into existence at the transition curve $\alpha = 1 - F/F_2$ through a Hopf bifurcation, as noted before.
For $\alpha < 1 - F/F_2$ and $F > F_1$, a single stable limit cycle exists (blue stipples in~Fig.~\ref{f-unfolding}). 

For $\alpha = 0$, it is therefore plausible that two stable limit cycles can coexist for $F> F_3, \,  \gamma_1(F_3) = 0$, assuming that this value is defined. It is usually impossible to obtain detailed analytical information about $\gamma_1$ and therefore about $F_3$. The authors in \cite{VanKekem2018travelling} use numerical methods to trace these curves. Alternatively, using the approach outlined here it is straight forward to compute $\gamma_1'(F_1)$ from the known normal form coefficients (see Section~\ref{sec-4-normal-form}). The tangent line approximation of $\gamma_1$ at $\alpha = \alpha_0$ (magenta line in Fig.~\ref{f-unfolding}) can thus be found and $F_3$ can be approximated by $F_3^\ast = F_1 - F_1/\gamma_1'(F_1)$, shown as a magenta circle in that figure. The approximation is expected to be more accurate if $F_1$ and $F_2$ are relatively close, such that $\alpha_0 = (F_2-F_1)/F_2$ is small.

\subsection{Numerical Results}

We have carried out numerical searches for coexisting stable attractors. For various site numbers $N$, the exact Hopf bifurcation values $F_1, \, F_2$ were computed as well as the approximate N-S bifurcation value $F_3^\ast$, using a linear approximation. Starting with a value $F$ that is somewhat larger than $F_3^\ast$, we then computed 100 numerical solutions with initial data  equal to $F \mathbf{e}$ plus a random normally distributed perturbation, for $0 \le t \le T = 1000$. At this $T$, solutions typically settled into one of several (usually two) stable limit cycles, which could be characterized by their spatial period. Lowering $F$ and following these stable limit cycles, approximate values $\tilde F_3 \approx F_3$ for the N-S bifurcations were found, and by increasing $F$, an approximate value $\tilde F_4$ was obtained up to which two stable limit cycles can be usually be observed.

In Table~2, we show results of these experiments for various site numbers $N$. In addition to the bifurcation parameters $F_i$, etc. we also give the spatial periods $m_1, \, m_2$ for the two limit cycles which appear as $F$ increases past $\tilde F_3$.

\begin{table}
\begin{center}
\begin{tabular}{c| c|c|c|c|c|c|c}
    $N$ & $F_1$ & $m_1$ & $F_2$ & $m_2$ & $F_3^\ast$ & $\tilde F_3$ & $\tilde F_4$ 
    \\
    \hline
    12 & 1 & 4 & 1 & 6 & 1 & 1 & $>2$\\
    14 & .8901 & 7 & 1.1820 & 14 & 1.5206 & not observed & not observed \\
    18 & .8982 & 9 & 1 & 6 & 1.1892 & not observed & not observed  \\
    22 & .9076 & 22 & .9343& 11 & .9915 & .996 & $>4$\\
    28 & .8901 & 14 & .9457 & 28 & 1.0293 & 1.072 & $>3$ \\
    36 & .8982 & 9 & .9025 & 36 & .9094 & .904 & $>2$
\end{tabular}
\end{center}
\label{table:limit.cycles}
\caption{Multiple stable limit cycles are expected for $F \gtrapprox F_3^\ast$ and are found numerically for $\tilde F_3 \le F \le \tilde F_4$. Limit cycles may be characterized by their spatial periods $m_1, \, m_2$.}
\end{table}

The results show that if $F_2$ is close to $F_1$ (i.e. if $\alpha_0$ is small), then the approximation $F_3^\ast$ is also close to $F_2$ and two stable coexisting limit cycles can be observed ($N = 12, \, 22, \, 28, \, 36$). This happens generally whenever $N$ is sufficiently large, since then the eigenvalues of $FA - I$ are more closely spaced. On the other hand, for relatively small $N$ ($N = 14, \, 18$), we find that $F_2$ is substantially larger than $F_1$, therefore $\alpha_0$ is relatively large, and it is not clear if two stable limit cycles coexist for some $F > F_2$. 

If $N$ is sufficiently large and has many small divisors, then more than two stable limit cycles may be observed for moderate $F$. For example, if $N = 144$ and $F = 2$, one can observe four different stable limit cycles, with spatial periods $m \in \{9, \, 24, \, 36, \, 144\}$. In fact, for any site number $N$ and $F$ near 1, there must be limit cycles with spatial periods for all divisors of $N$, since L96-like systems with site numbers that divide $N$ can be embedded periodically into a system with site number $N$.

\section{General Existence and Stability}
\label{sec-5}

We now move on to a second type of generalization of the L96 system. One can modify the system to include site-specific dissipation, advection, and time-dependent forcing terms in order to consider, for example, changes in atmospheric dynamics and predictability over ocean and land. We may therefore consider the general  system
\begin{equation}
    \dot{\mathbf{x}}(t) = C G(\mathbf{x}(t)) - B \mathbf{x}(t) + \mathbf{F}(t),
    \label{eq:inhomogeneous}
\end{equation}
where $C, \, B$ are diagonal matrices and $\mathbf{F}(\cdot)$ is now a vector valued function. Another generalization, where forcing is time-periodic and possibly state dependent, was explored in \cite{lucarini2011}.

\subsection{Stationary Solutions}

We begin by looking at time-independent solutions of Eq.~(\ref{eq:inhomogeneous}), that is solutions of the system $0 = C G(\mathbf{x}) - B \mathbf{x} + \mathbf{F}$. We assume that $C$ and $B$ are positive diagonal matrices and $G$ is a quadratic map. If $G$ is also energy-preserving, then the resulting algebraic system always has at least one solution as the following result shows. We may assume without loss of generality that $C$ is the identity matrix.

\begin{prop}
Let $B$ be a positive diagonal $N \times N$ matrix and assume that $G$ is energy-preserving. Then for any $\mathbf{F} \in \mathbb{R}^N$ there exists a solution $\mathbf{x}$ of  
\begin{equation}
0 = G(\mathbf{x}) - B \mathbf{x} + \mathbf{F}
\label{e-stationary}
\end{equation}
and all such solutions satisfy
\begin{equation}
\|B^{1/2}\mathbf{x}\| \le \|B^{-1/2} \mathbf{F} \| \, .
\label{e-apriori}
\end{equation}

\end{prop}

\begin{proof}

Taking the scalar product with $\mathbf{x}$, using the fact that $G$ is energy-preserving, and applying the Cauchy-Schwarz inequality results in
$$
\begin{aligned}
0 &= -\mathbf{x}^{\text{T}}  B \mathbf{x} +  \mathbf{x}^{\text{T}}    \mathbf{F} =
-\| B^{1/2} \mathbf{x}\|^2 + \mathbf{x}^{\text{T}} \mathbf{F} \\
&\le -\| B^{1/2} \mathbf{x}\|^2 + \| B^{1/2} \mathbf{x}\|
\| B^{-1/2} \mathbf{F} \| \, .
\end{aligned}
$$
Rearranging this implies the estimate Eq.~(\ref{e-apriori}). 

To show existence, we use a degree argument for continuous maps; see e.g. \cite{allgowerGeorg2012}. Consider the function \begin{equation}
\Phi:[0,1] \times \mathbb{R}^N \, \to \, \mathbb{R}^N, \quad \Phi(t, \mathbf{z}) =  G(\mathbf{z}) - B\mathbf{z} + t \cdot \mathbf{F} \, .
\label{e-homotopy}    
\end{equation}

It is jointly continuous in $t$ and $\mathbf{z}$. Let $K = \{\mathbf{z} \, : \, 
\| B^{1/2}\mathbf{z}\| \le \|B^{-1/2}\mathbf{F}\| + 1 \} \subset \mathbb{R}^N$. Then Eq.~(\ref{e-apriori}) implies that the equation $\Phi(t, \mathbf{x}) = 0$ does not have any solutions on the boundary
$\partial K$ for $0 \le t \le 1$.  Therefore the mapping degree 
$\text{deg}(\Phi(t, \cdot), K, 0)$ is constant for $0 \le t \le 1$. Since for $t = 0$ the only solution is $\mathbf{x} = 0$, due to estimate Eq.~(\ref{e-apriori}), $\text{deg}(\Phi(0, \cdot), K, 0)= \text{deg}(\Phi(1, \cdot), K, 0)= 1$, and the equation $\Phi(1,\mathbf{x}) = G(\mathbf{x}) - B\mathbf{x} + \mathbf{F} = 0$ also has a solution in $K$. 
\end{proof}

\begin{figure}[ht]
    \centering
        \includegraphics[width = 0.8\textwidth]{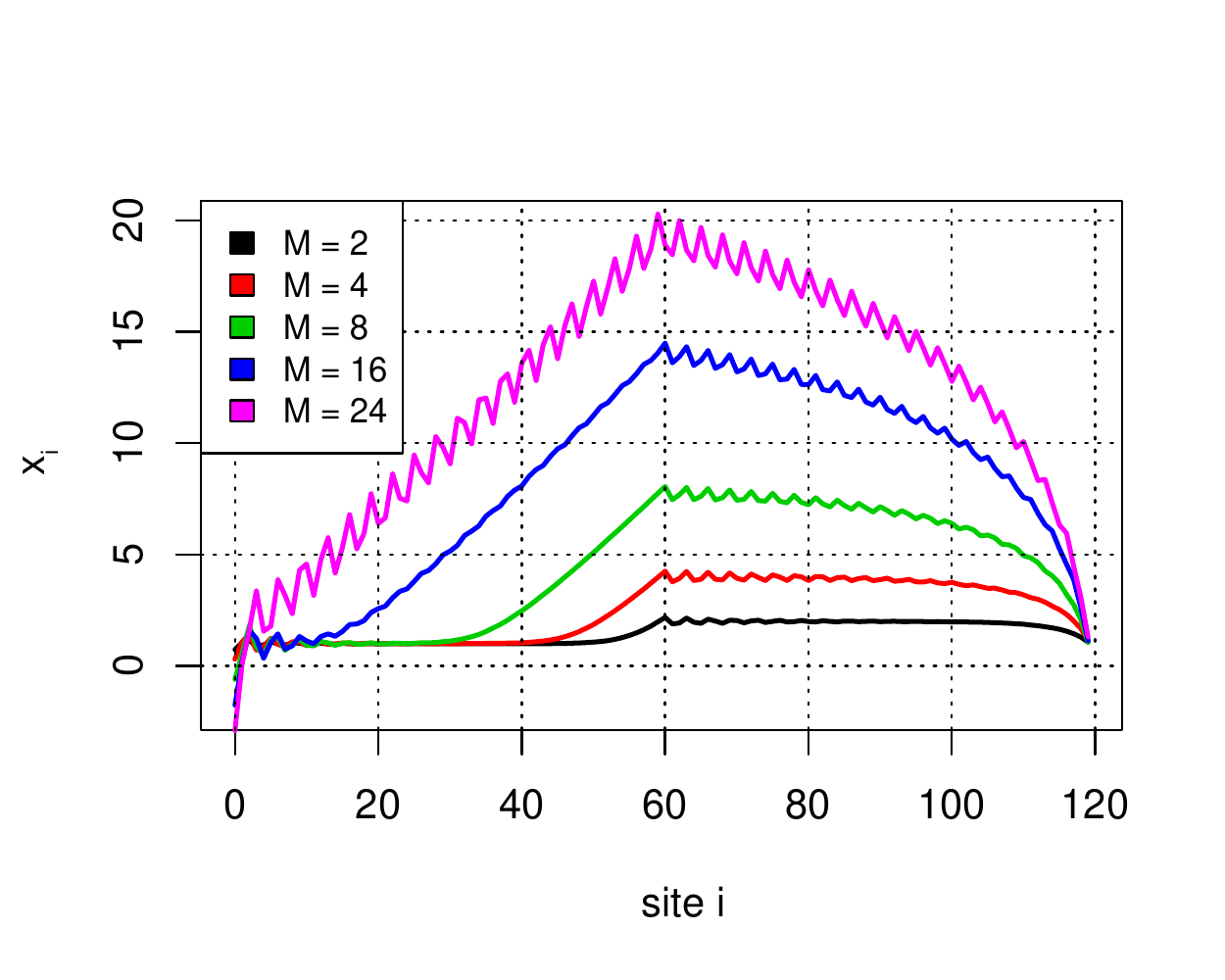}
        \caption{Stationary solutions of Eq.~(\ref{e-stationary}) with $G = G_L, \, C = B = I$  		and inhomogeneous forcing $\mathbf{F}$ for $N = 120$ sites. Here  $F_i = 1$ for 
        $0 \le i  < N/2$ and $F_i = M$ for $i \ge N/2)$.}
\label{fig-stationary}
\end{figure}

It should be noted that solutions of Eq.~(\ref{e-stationary}) need not be unique. Indeed, it is known that for the L96 system there may be many different solutions  for sufficiently large negative $F$, assuming that $N$ is divisible by a large power of 2  (\cite{VanKekem2019symmetries}). These solutions appear in a cascade of pitchfork bifurcations. In general, systems of $N$ quadratic equations may have up to $2^N$ solutions.

Solutions of Eq.~(\ref{e-stationary}) may be computed by choosing a sequence $0 = t_0 < t_1 < t_2 < \dots < t_K = 1$ and finding solutions $\mathbf{x}_i$ of $\Phi(t_i, \mathbf{x}_i) = 0$, with $\Phi$ as in Eq.~(\ref{e-homotopy}). Each solution $\mathbf{x}_i$ is computed using a Newton iteration with $\mathbf{x}_{i-1}$ as a starting value. A thorough discussion of such homotopy methods may be found in \cite{allgowerGeorg2012}. 

This is illustrated in Fig.~\ref{fig-stationary}.  Here $N = 120$ and $\mathbf{F}$ is non-constant, with $F_i = 1$ for $i < N/2$ and $F_i = M$ for $i \ge N/2$, for five different choices of $M$. Solutions were found by choosing $t_i = i \delta$, with $\delta = 10^{-1}$ for $M = 2$ and $\delta = 10^{-3}$ for $M = 24$. For small and moderate $M$, the solutions satisfy $x_i \approx F_i$ except to the left of the transition points $i = 0, i = N/2$. For all $M$, solutions are maximal near $i = N/2$ and minimal near $i = 0$, with roughly monotone transitions and superposed oscillations of period 3. It also appears that the maxima of these solutions are  close to $M$ for small and moderate $M$, but  not for large $M$. It is unclear how to explain this behavior. Note however that this is only one stationary solution for each $M$ out of possibly very many others. We did not investigate the dynamic stability of these stationary solutions, but suspect that they become unstable as $M$ increases. 

\subsection{Dynamic Problem}

An existence result for solutions of the 
time-dependent problem Eq.~(\ref{eq:inhomogeneous}) can be shown under fairly general conditions.

\begin{prop}
Let $C$ be a positive diagonal $N \times N$ matrix, $B$ an arbitrary $N \times N$ matrix, and assume that $G$ is an energy preserving mapping. Then for any $\mathbf{x}_0 \in \mathbb{R}^N$ and any continuous curve $\mathbb{R} \ni t \mapsto \mathbf{F}(t)$ there exists a global solution $\mathbb{R} \ni t \mapsto \mathbf{x}(t)$ of Eq.~(\ref{eq:inhomogeneous}) 
with $\mathbf{x}(0) = \mathbf{x}_0$. 
\end{prop}

\begin{proof}
By standard existence and uniqueness results, there exists a unique solution of Eq.~(\ref{eq:inhomogeneous}) in a maximal open interval $I = (t_0,t_1)$ that contains $t = 0$. To prove that the solution exists on $\mathbb{R}$, it is sufficient to show that it remains bounded on any such open interval $I$. Consider first the set $[0,t_1)$ and assume that $t_1$ is finite. Take the scalar product of Eq.~(\ref{eq:inhomogeneous}) with $C^{-1}\mathbf{x}(t)$, use the assumption that $G$ is energy-preserving, and integrate over $[0,t] \subset [0,t_1)$. The result is the identity
$$
\frac{1}{2} \left( \mathbf{x}(t)^{\text{T}}C^{-1}\mathbf{x}(t) - \mathbf{x}_0^{\text{T}}C^{-1} \mathbf{x}_0\right) =
\int\limits_0^t \left( -\mathbf{x}^{\text{T}}(s)C^{-1}B\mathbf{x}(s) + \mathbf{x}^{\text{T}}(s)C^{-1}\mathbf{F}(s) \right) \, ds \, .
$$
Since $C$ is positive definite, this implies after some standard estimates
$$
\|\mathbf{x}(t)\|^2 \le c_0 + \int\limits_0^t c_1 \|\mathbf{x}(s)\|^2 \, ds +
\int\limits_0^t c_2 \|\mathbf{x}(s)\| \, ds \le c_0 + \int\limits_0^t \left( c_1 + \frac{c_2}{2} \right)  \|\mathbf{x}(s)\|^2 \, ds + \frac{c_2t_1}{2}
$$
where the $c_i$ are suitable constants and in particular $c_2$ is a multiple of $\max_{0 \le s \le t_1} \|\mathbf{F}(s)\|$. By the Gronwall lemma, $\|\mathbf{x}(t)\| \le const.$ for all $t < t_1$, with the constant not depending on $t$. The solution therefore may be continued for all positive $t$. To show that the solution also may be continued for all negative $t$, replace $t$ with $-t$, $B$ with $-B$, and $G$ with $-G$ (which is also energy-preserving), and apply the same argument.    

\end{proof}

\subsection{Global Stability for Small Forcing}

Here, we consider the non-stationary inhomogeneous problem Eq.~(\ref{eq:inhomogeneous}) for constant forcing terms $\mathbf{F}$. If $B$ and $C$ are positive diagonal matrices and if $G$ is quadratic, a straightforward exercise shows that the zero solution is asymptotically stable for zero forcing. In that case, the mapping $\mathbf{x} \mapsto CG(\mathbf{x}) - B\mathbf{x}$ is locally invertible near the origin, and therefore, by a standard perturbation argument, small solutions are also asymptotically stable for appropriately small forcing. Here we show that if  $G$ is in addition energy-preserving, then solutions for sufficiently small forcing  are actually globally asymptotically stable and in particular unique. 

\begin{prop}
Let $B$ and $C$ be positive diagonal $N \times N$ matrices and let $G$ be a quadratic and energy-preserving mapping. Then there exists $\varepsilon > 0$ such that for all constant $\mathbf{F}$ with $\|\mathbf{F}\| < \varepsilon$ and all $\mathbf{x}_0$ the solution of Eq.~(\ref{eq:inhomogeneous}) with $\mathbf{x}(0) = \mathbf{x}_0$ converges to a unique solution $\mathbf{x}_\infty$ of the stationary problem Eq.~(\ref{e-stationary}).  
\end{prop}

\begin{proof}
Let $\mathbf{x}_\infty$ be any solution of Eq.~(\ref{e-stationary}). Set $\mathbf{y}(t) = \mathbf{x}(t) - \mathbf{x}_\infty$. Since $G$ is quadratic, then due to Eq.~(\ref{e-quadratic.1}) the function $\mathbf{y}(\cdot)$ satisfies
\begin{equation}
    \dot{\mathbf{y}}(t) = C G(\mathbf{y}(t)) + C A[\mathbf{x}_\infty]\mathbf{y}(t) - B \mathbf{y}(t) \, .
\label{e-y.inhomogeneous}
\end{equation}
Here $A$ is a linear matrix valued map that depends only on $G$. Now choose $\delta > 0$ such that $-C^{-1}B + A[\mathbf{x}]$ is still negative semidefinite for all $\|\mathbf{x}\| < \delta$. Multiply Eq.~(\ref{e-y.inhomogeneous}) with $C^{-1}\mathbf{y}(t)$ and use the assumption that $G$ is energy-preserving. It follows that if $\|\mathbf{x}_\infty\| < \delta$, then for some $\lambda > 0$
\begin{equation}
\frac{1}{2}\frac{d}{dt} \mathbf{y}^{\text{T}}(t)C^{-1}\mathbf{y}(t) \le - \lambda
\mathbf{y}^{\text{T}}(t)C^{-1}\mathbf{y}(t) .
\label{e-y.stable}
\end{equation}
This inequality of course implies that $\mathbf{y}(t) = \mathbf{x}(t) - \mathbf{x}_\infty \to 0$ and that $x_\infty$ is unique. By Eq.~(\ref{e-apriori}), there is a constant $c_3$ such that $\|\mathbf{x}\| \le c_3 \|\mathbf{F}\|$ for all solutions of the stationary problem. Thus if $\|\mathbf{F}\| < \varepsilon = \frac{\delta}{c_3}$, then $\|\mathbf{x}_\infty\| < \delta$ and therefore $\mathbf{x}(t) \to \mathbf{x}_\infty$ no matter what $\mathbf{x}_0$ is. 

\end{proof}

\subsection{Numerical Examples}
Recall how to rescale the system~(\ref{e-rescaling}) so that it has the form of Eq.~(\ref{e-lorenz-1v}). If $\tilde{x}_i(t)$ is the solution of Eq.~(\ref{e-rescaling}) with initial state $\tilde{\mathbf{x}}_0$, and $x_i(t)$ is the solution of Eq.~(\ref{e-lorenz-1v}) with $F = \alpha\gamma / \beta^2$ and initial state $\alpha\beta^{-1} \tilde{\mathbf{x}}_0$, then $\tilde{x}_i(t) = \alpha\beta^{-1} x_i(\beta^{-1} t)$.

We now consider inhomogeneous systems of the form
\begin{equation}
    \dot{x}_i = \alpha_i x_{i-1}(x_{i+1} - x_{i-2}) - \beta_i x_i + \gamma_i
    \label{eq:inhomogeneous-component}
\end{equation}
with
$$
(\alpha_i,\beta_i,\gamma_i) = 
\begin{cases}
(1,1,2), \text{ for } i \leq N/2\\
(\alpha,\beta,\gamma), \text{ for } i > N/2
\end{cases}.
$$
This means that at sites $i < \frac{N}{2}$ (``left half''), the system is expected to behave as shown in Fig.~\ref{fig-hov-f2}. At the remaining sites (``right half'') the scaling argument shows that, qualitatively, the dynamics should be approximately the same as the dynamics of the standard system with $F=\alpha \gamma/\beta^2$, multiplied by $\beta/\alpha$ and sped up by a factor $\beta^{-1}$. Although this scaling argument cannot be applied to sites very near the boundaries between the regions, some insight into the dynamics at the boundaries can be gained by considering the behavior on either side.

\begin{figure}[ht]
    \centering
    \includegraphics[width=0.49\textwidth]{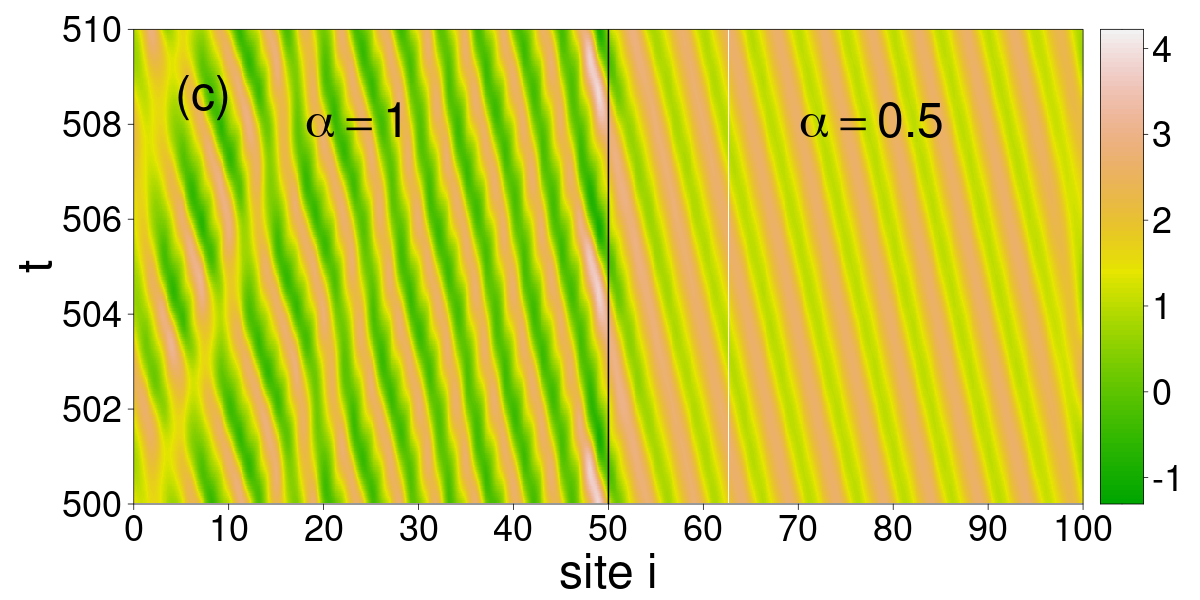}
    \includegraphics[width=0.49\textwidth]{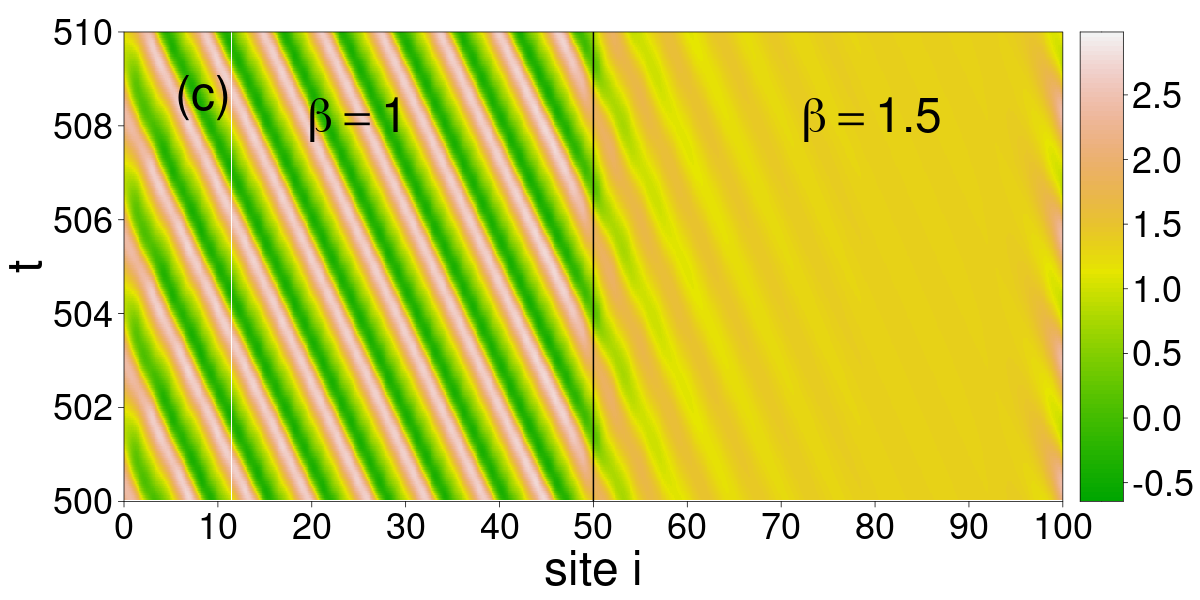}
    \caption{Hovmoeller plots showing inhomogeneous advection and dissipation as described in Eq.~(\ref{eq:inhomogeneous-component}), for $N = 100$. Both panels use the same parameters $(1,1,2)$ in the left half, but solutions have very different behavior. \textbf{Left:} Sites in the right half have parameters $(0.5,1,1)$. Smaller advection in the right half leads to smaller spatial amplitudes. Perturbations are seen to travel to the right.   \textbf{Right:} Sites in the right half have parameters $(1,1.5,1)$. Larger dissipation in the right half leads to nearly constant solutions over a substantial range of sites.}
    \label{fig-inhom_dis_f}
\end{figure}

Figure~\ref{fig-inhom_dis_f} shows two Hovmoeller plots for systems with $N = 100$ and $500 \leq t \leq 510$, with random initial values at $t=0$. Note first that the parameters for the left halves of each panel are the same, but the solution behavior is clearly quite different, showing a slowly moving solution with spatial period near $m = 4$ in the left panel and a faster moving solution with period near $m = 8$ in the right panel. The left panel shows the effect of a lower advection in the right half while the right panel shows the effect of increased dissipation in the right half. Some expected effects can be noticed right away. For example, the left panel shows waves of larger amplitude with midpoint shifted upward in the left half, which leads to large fluctuations at its downstream boundary $i = 50$ around $t=503$ and $t=506$ due to the large positive advection term there. In both halves of the left panel, perturbations at the downstream boundary are seen to travel to the right, as predicted by the discussion in Section~\ref{sec-3-waves}. The right panel shows waves travelling to the left in the left half, while in the right half the effective $F = \alpha \gamma/\beta^2 = 8/9$ now is less than the Hopf bifurcation value and the solution is close to constant for a substantial range of sites.  

We also noticed some unanticipated effects, especially for larger forcing. For example, a very small parameter difference between the regions may completely change which attractor the system reaches.

\section{Conclusions and Open Questions}
\label{sec-6}

In this paper we have described several analytical and numerical tools for the study of the L96 system. These tools are equally useful for several possible generalizations  of the L96 system, which are also proposed here. Specifically, we have exploited the fact that linearizations of the system about constant states are given by circulant matrices to introduce a graphical approach for the study of its spectrum and in order to demonstrate that normal forms of the system at bifurcation points may be computed analytically. This allows one to compute, for example, the first Lyapunov exponent for the L96 system, proving that the Hopf Bifurcation must be supercritical.

Further, we have classified all advection terms that share essential features with the original Lorenz term. Our analysis identifies the original L96 advection term $G_L$ as the simplest such term that produces the signature complex wave behavior. While this version of the system appears to have received the most attention in the literature so far, we have identified and classified a whole class of related advection terms that can be expected to lead to similar dynamics. 

We have also considered possible modifications of the L96 system by allowing site-dependent advection, dissipation, and forcing. We have presented basic results for the existence and stability  of solutions for time-independent and time-dependent problems. We have also explored examples numerically and explained some of the new  phenomena using site-dependent scaling arguments.

The possible modifications presented in this paper open up many avenues of investigation. It seems natural to consider systems with alternative advection terms for large forcing in order to discover whether, and in which cases, the rich dynamic behavior of the L96 system still results. A concrete problem in this domain is to compare two systems having different advection terms but the same linearization and thus the same eigenvalue curve, and to study how they differ with increasing forcing.

There is also a variety of questions related to bifurcations to explore. Do multiple stable limit cycles appear also for other advection terms that admit Hopf bifurcations, followed by chaotic behavior for larger forcing? Do these phenomena persist universally in spite of the large flexibility in the set of quadratic, equivariant, energy-preserving, localized advection terms that has been identified here? It was observed numerically that in the case of large site numbers with many small divisors, more than two stable limit cycles may exist for $F$ near 1. Is there a rigorous explanation for this phenomenon?

A broader range of questions can be asked for inhomogeneous systems. In the case of inhomogeneous advection terms, constant equilibrium solutions can still be found, but the application analysis becomes more involved because linearizations no longer result in circulant matrices. One should ask about the stability of stationary solutions in the case of inhomogeneous forcing such as e.g. in Fig.~\ref{fig-stationary}. How do inhomogeneous advection/dissipation/forcing affect attractors and their basins of attraction? And finally, are there better ways of characterizing and identifying attractors than have been found numerically?

The Lorenz '96 system exhibits a wide range of complex dynamical behavior. Together with its many possible modifications, it is a showcase for a variety of interesting phenomena and at the same time a rich source of mathematical challenges. It is our hope that this paper will ultimately contribute to better attention to this subject, especially from the mathematics community.

\section*{Appendix A: Discrete Fourier Transform}

Here we collect some basic facts relating the L96 system to the discrete Fourier transform. The same arguments can be applied to any $\mathscr{G}$-map. Let 
\begin{equation}
\mathcal{F}_N = \frac{1}{\sqrt{N}} \left( e^{2 \pi i k \ell/N}\right)_{0 \le k, \ell < N} = \frac{1}{\sqrt{N}} \left( \omega_N^{ k \ell}\right)_{0 \le k, \ell < N}
\label{e-FourierMatrix}
\end{equation}
be the normalized symmetric Fourier matrix, where as before $\omega_N = e^{2  \pi i/N}$. We denote the columns of $\mathcal{F}_N$ by $\mathbf{q}_\ell$ (see also Eq.~(\ref{e-Fourier1})). The conjugate matrix $\mathcal{F}_N^\ast$ has columns $\mathbf{q}_{N-\ell}$ and is the inverse of $\mathcal{F}_N$.

Let $A$ be any real circulant $N \times N$ matrix, with top row $(c_0, c_1, \dots, c_{N-1})$.  Let $p_A(z) = \sum c_jz^j$ be the associated Laurent polynomial, where the summation extends over $-N/2 < j \le N/2$. Then $A$ has eigenvalues $p_A(\omega_N^\ell)$ with right eigenvectors $\mathbf{q}_\ell$. The left eigenvector of $A$ for the eigenvalue $p_A(\omega_N^\ell)$ is the $\ell$-th row of $\mathcal{F}_N^\ast$ or $\mathbf{q}_{N-\ell}^{\text{T}}$. The matrix $\mathcal{F}_N^\ast A \mathcal{F}_N$ is diagonal with entries $p_A(\omega_N^\ell)$.

We now describe the action of a bilinear $\langle \rho \rangle$-equivariant map $\mathcal{B}$, defined in Eq.~(\ref{e-quadratic}), in terms of the Fourier matrix. As before, all indices are understood modulo $N$. Let $Q$ be the associated symmetric matrix, defined in Eq.~(\ref{e-define-Q}), extended $N$-periodically for negative indices. 

\begin{prop}
\label{prop-Fourier}
For $z, \, w \in \mathbb{C}, \, z \ne 0 \ne w$ let 
\begin{equation}
P_{\mathcal{B}}(z,w) = \sum_{-N/2 < r, \, s \le N/2}z^rQ_{rs}w^s.
\label{e-BPolynomial}
\end{equation}
(a) Let $0 \le k, \, \ell < N$, then 
\begin{equation}
\mathcal{B}(\mathbf{q}_k, \, \mathbf{q}_\ell) = \frac{P_\mathcal{B}(\omega_N^k, \omega_N^\ell)}{\sqrt{N}} \mathbf{q}_{k + \ell} \quad
\text{and} \quad 
G(\mathbf{q}_k) =  \frac{P_\mathcal{B}(\omega_N^k,\omega_N^k)}{2 \sqrt{N}} \mathbf{q}_{2k} .
\label{e-QuadraticPolynomial}
\end{equation}
(b) In particular, for the L96 case,
\begin{equation}
P_L(z,w) = z^{-1}(w - w^{-2}) + w^{-1}(z - z^{-2})
\label{e-LPolynomial}
\end{equation}
and
\begin{equation}
\mathcal{B}_L(\mathbf{q}_k, \, \mathbf{q}_\ell) = \frac{P_L(\omega_N^k, \omega_N^\ell)}{\sqrt{N}} \mathbf{q}_{k + \ell} \,, \quad
G_L(\mathbf{q}_k) = \frac{1 - \omega_N^{-3k}}{\sqrt{N}} \mathbf{q}_{2k} .
\label{e-FourierQuadratic1}
\end{equation}
(c) Let $A = \mathcal{B}(\cdot, \mathbf{e})$ be the linearization of a quadratic map $G$ about the unit vector $\mathbf{e}$ and let $p_A$ be the associated Laurent polynomial with coefficients coming from the top row of $A$. Then for all non-zero $z \in \mathbb{C}$ 
\begin{equation}
p_A(z) = P_\mathcal{B}(z,1) \, .
\label{e-LinearizationPolynomial}
\end{equation}  \\
(d) Let $\mathbf{y} \in \mathbb{R}^N$. Then
\begin{equation}
\left( \mathcal{F}_N^\ast G(\mathcal{F}_N \mathbf{y}) \right)_\ell = \sum_{j = 0}^N 
\frac{P_\mathcal{B}(\omega_N^j, \omega_N^{\ell - j})}{2\sqrt{N}} y_j y_{\ell - j} \, .
\label{e-FourierQuadratic2}
\end{equation}
\end{prop} 
These formulae can be proved by direct computation. 

Finally it should be noted that if $N$ is a multiple of 3 and $k$ is a multiple of $N/3$, then $G_L(\mathbf{q}_k) = 0$, since then $\omega_N^{-3k} = 1$. More generally it can be shown that $P_L(z,w) = 0$ on the torus where $|z| = |w| = 1$ if and only if $z$ and $w$ are of the form $e^{\pm 2 \pi i/3}$.  

\section*{Appendix B: Numerical Considerations}

The papers \cite{Lorenz96} and \cite{LorenzEmanuel98} used the classical fourth order Runge-Kutta scheme (RK4) with fixed temporal step width $\Delta t = 0.05$ for the numerical solution of Eq.~(\ref{e-lorenz-1}). While this method is straightforward to implement and makes all computations replicable, it only has mid-level accuracy.

\begin{wrapfigure}{r}{0.45\textwidth}
    \centering
    \vspace{-0.75cm}
    \includegraphics[width = 0.43\textwidth]{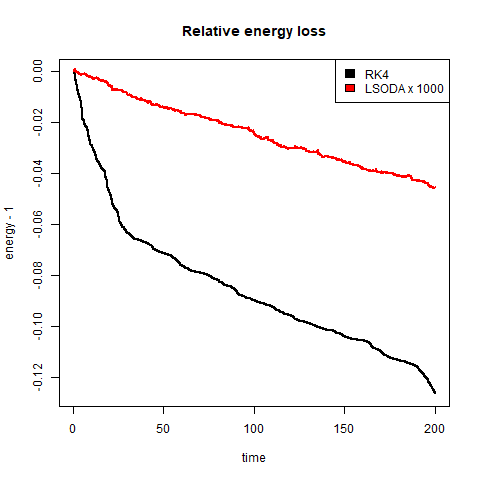}
    \setlength{\abovecaptionskip}{-0.25cm}
    \setlength{\belowcaptionskip}{-0.25cm}
    \caption{Relative energy loss for RK4 and \texttt{lsoda} (multiplied by $10^3$), for $N=36$, $\Delta t = 0.05$, and initial energy 400.}
    \label{f-energy}
\end{wrapfigure}

In this paper, we use a high accuracy ODE solver with variable step width instead, described in the article \cite{lsoda1997} and implemented as \texttt{lsoda} in the \texttt{R} package \texttt{deSolve} (see \cite{deSolve}). To test the accuracy of the solver, we solved the advection-only version Eq.~(\ref{e-lorenz-general-adv}) with the L96 advection term $G_L$.  Then the energy $\sum_i x_i^2$ is exactly conserved.  By rescaling, if $t \mapsto \mathbf{x}(t)$ is a solution, then so is $t \mapsto \lambda \mathbf{x}(\lambda t)$; that is, for larger energy value the system has essentially the same trajectories on a faster time scale. The performance of the RK4 solver with a fixed time step can then be expected to  deteriorate, while the performance of the \texttt{lsoda} solver should be less affected, since it has  automatic stepsize control and is designed to handle stiff equations. 

As expected, for both solvers the numerical energy is not exactly constant. For small values of the energy, e.g. near 1, both solvers nearly conserve energy for long time intervals.  If the energy of the initial value becomes larger, solutions lose energy for both solvers, but solutions found with LSODA lose energy at a much slower rate. The papers \cite{abramov2004, abramov2007, abramov2008, gallavotti2014} use a variety of numerical solvers and emphasize the need to use high order adaptive methods, as measured e.g. by observed energy loss for inviscid systems. 

As an example, Fig.~\ref{f-energy} plots the energy of numerical solutions with initial data $x_j(0) = c(1 + \sin 2 \pi \frac{j}{N})$ for $ 0 \le j < N = 36$, where $c$ is chosen such that the energy equals 400 at $t = 0$. The solution computed with \texttt{lsoda} loses about $2 \cdot 10^{-5} \%$ of its energy per unit time  while the solution from RK4 with $\Delta t = 0.05$ loses around $4 \cdot 10^{-2} \%$ of the energy per unit time.

\section*{Appendix C: Nondissipative Systems}

In this section we consider versions of Eq.~(\ref{e-lorenz-general}) that contain only advection terms; that is,
\begin{equation}
\dot{\mathbf{x}} = G(\mathbf{x})
\label{e-lorenz-general-adv}
\end{equation}
where $G$ is a $\mathscr{G}$-map.  The system~(\ref{e-lorenz-general-adv}) appears as a formal limit for $F \to \infty$, if $t$ and $\mathbf{x}$ are rescaled as $\tau = F^\gamma t, \, \mathbf{y}(\tau) = F^{-\gamma} \mathbf{x}(\tau)$ with $\gamma > \tfrac12$. For $\gamma = \tfrac23$ this rescaling was proposed in \cite{lorenz2005}. The limiting behavior as $F$ becomes large was studied systematically in \cite{gallavotti2014}. Rescaling results in the system
\begin{equation}
\frac{d}{d\tau} \mathbf{y} = G(\mathbf{y}) - F^{-\gamma} \mathbf{y} + F^{1 - 2 \gamma}  \, .
\label{e-lorenz-general-scaled}
\end{equation}
The system~(\ref{e-lorenz-general-adv}) resulting from sending $F \to \infty$  is also known as \textbf{inviscid Lorenz 96} system and was studied in \cite{abramov2004}. In \cite{orszag1980}, the inviscid system~(\ref{e-lorenz-general-adv}) was introduced for the symmetric case $G = G_3 + \tilde G_3$. Another rescaling was proposed in \cite{blender2013}, with the goal of deriving and investigating a continuous long wave approximation, which is effectively a quasilinear second order partial differential equation without dissipation and with conserved energy in the case of zero forcing.       

Here are some observations for the cases $G = G_3$ (the original Lorenz system) and the symmetric version $G = G_3 - \tilde G_3$. 

If $G = G_3$ or $G = G_3 - \tilde G_3$ and if  the number of sites $N$ is a multiple of 3, initial data that are periodic with period 3 result in constant solutions. If $G = G_3 - \tilde G_3$ and the number of sites is even, initial data that are periodic with period 2 also result in constant solutions. For $G = G_3 - \tilde G_3$ in the case where $N$ is a multiple of 6, there is therefore a four-dimensional variety of initial data with period 6 that result in constant solutions. 

In the case $G = G_3 - \tilde G_3$, the sum over all sites $\sum_i x_i(t)$ is constant along solutions. If $N$ is even, then the energy $\sum_{j} x_{2j}(t)^2$ at the even-numbered sites and the energy at the odd-numbered sites are conserved separately. If $N$ is a multiple of 3, then the sums over the sites that are three spaces apart, i.e. $\sum_j x_{3j+k}(t)$  for $k = 0, \,1, \, 2$ are separately conserved. 

For the case $G = G_3 - \tilde G_3$, if $N = 4$, the system is Hamiltonian. In this case there are the two conserved quantities $x_0(t)^2 + x_2(t)^2 = \rho_0^2$ and  
$x_1(t)^2 + x_3(t)^2 = \rho_1^2$.  Introducing polar coordinates $x_0(t) + i \, x_2(t) = \rho_0 e^{i \cdot \alpha_0(t)}$ and $x_1(t) + i \, x_3(t) = \rho_1 e^{i \cdot \alpha_1(t)}$, one obtains after some algebra the system 
\begin{eqnarray}
\alpha_0'(t) &=& -\rho_1 \sqrt{2}
\cos \left( \alpha_1(t) + \frac{\pi}{4} \right)\\
\alpha_1'(t) &=& \rho_0 \sqrt{2}
\cos \left( \alpha_0(t) + \frac{\pi}{4} \right) .
\end{eqnarray}
This is a Hamiltonian system with $H(\alpha_0, \, \alpha_1) = \sqrt{2} \left(\rho_0 \sin (\alpha_0 + \frac{\pi}{4})  + \rho_1  \sin (\alpha_1 + \frac{\pi}{4})\right) $. 

For the case $G = G_3 - \tilde G_3$ and $N = 6$, all solutions may be written in the form
$ x_j(t) = c_j + \frac{1}{2}(-1)^j y_j(t)$ for $j = 0, \, 1, \, 2$ and $x_j(t) = c_{j-3} + \frac{1}{2}(-1)^j y_{j-3}(t)$, for $j = 3, \, 4, \, 5$, 
for suitable functions $y_0, \, y_1, \, y_2$ and constants $c_j$. Let $\mathbf{ c} = (c_0, \, c_1, \, c_2)^{\text{T}}$ and $\mathbf{y} (t) = (y_0(t), \, y_1(t), \, y_2(t))^{\text{T}}$. Then a direct calculation shows that $\mathbf{y}(t)$ satisfies the linear system 
$\dot{\mathbf{y}}(t) = \mathbf{c} \times \mathbf{y}(t)$.

\bigskip

\bigskip
\textit{Both authors contributed to the conception, design, analysis, computation, and writing of this study. This article is based on the first author's undergraduate senior thesis.}


\newpage

\bibliographystyle{plain}
\bibliography{lorenz96}

\end{document}